\documentclass{amsart}

\title[The full Kostant-Toda hierarchy on the positive flag variety]
{The full Kostant-Toda hierarchy on the positive flag variety}

\author{Yuji Kodama and Lauren Williams} 
\date{\today}
\thanks{The first author was partially
supported by NSF grant DMS-1108813.  The second author was 
partially supported by an NSF CAREER award.}
\address{Department of Mathematics, Ohio State University,
Columbus, OH 43210}
\email{kodama@math.ohio-state.edu}
\address{Department of Mathematics, University of California,
Berkeley, CA 94720-3840}
\email{williams@math.berkeley.edu}
\subjclass[2000]{}

\usepackage{url}
\usepackage{color}
\usepackage{rotating}

\countdef\x=23
\countdef\y=24
\countdef\z=25
\countdef\t=26

\def\tbox(#1,#2)#3{
\x=#1 \y=#2 
\multiply\x by 12 
\multiply\y by 12 
\z=\x \t=\y
\advance\z by 12 
\advance\t by 12 
\psline(\x,\y)(\x,\t)(\z,\t)(\z,\y)(\x,\y)
\advance\x by 6
\advance\y by 6 
\rput(\x,\y){{\bf #3}}}

\usepackage{amssymb}
\usepackage{graphicx}
\usepackage{epstopdf}
\usepackage{amscd}
\usepackage{appendix}
\usepackage{graphics}
\def\proof{\par{\it Proof}. \ignorespaces}
\def\endproof{{\ \vbox{\hrule\hbox{%
     \vrule height1.3ex\hskip0.8ex\vrule}\hrule }}\par}

\theoremstyle{definition}

\theoremstyle{remark}

\numberwithin{equation}{section}

\let\trueint=\int
\let\truesum=\sum
\def\int{\mathop{\textstyle\trueint}\limits}
\def\sum{\mathop{\textstyle\truesum}\limits}
\let\<=\langle
\let\>=\rangle

\def\RR{{\mathcal R}}

\def\SL{\mathrm{SL}}
\def\Sym{{\mathfrak S}}

\def\v{\mathbf{v}}
\def\w{\mathbf{w}}
\def\t{\mathbf{t}}
\def\0{\mathbf{0}}

\usepackage{amsmath, amsthm, amssymb, amsbsy,amsfonts,amscd}
\usepackage[mathscr]{eucal}
\usepackage{epsfig}
\usepackage{young}
\newtheorem{theorem}{Theorem}[section]

\newtheorem{definition}[theorem]{Definition}
\newtheorem{proposition}[theorem]{Proposition}
\newtheorem{lemma}[theorem]{Lemma}

\newtheorem{example}[theorem]{Example}
\newtheorem{corollary}[theorem]{Corollary}

\newtheorem{remark}[theorem]{Remark}

\usepackage{amsfonts}
\usepackage{xy}
\usepackage{amssymb}
\usepackage{amsmath}

\newcommand{\To}{\longrightarrow}

\newcommand{\R}{\mathbb R}
\newcommand{\QQ}{\mathsf Q}

\newcommand{\C}{\mathbb C}

\DeclareMathOperator{\constant}{constant}

\DeclareMathOperator{\convex}{Conv}
\DeclareMathOperator{\conv}{Conv}
\DeclareMathOperator{\CH}{Conv}
\DeclareMathOperator{\Perm}{Perm}

\DeclareMathOperator{\M}{\mathcal M}

\DeclareMathOperator{\diag}{diag}

\newcommand{\thmrefer}[1]{\renewcommand\thetheorem
  {\protect\ref{#1}}\addtocounter{theorem}{-1}}

\xyoption{all}
\CompileMatrices

\begin{document}

\begin{abstract}
We study some geometric and  combinatorial aspects of the solution to the 
full Kostant-Toda (f-KT) hierarchy,
when the initial data is given by 
an arbitrary point on the totally non-negative (tnn) 
flag variety of $\text{SL}_n(\R)$. 
The f-KT flows on the tnn flag variety are complete, 
and their asymptotics are completely determined
by the 
cell decomposition of the tnn flag variety given by Rietsch \cite{Rietsch}.
We define the f-KT flow on the weight space via
the moment map, and show that 
the closure of each f-KT flow forms an interesting convex polytope
generalizing the permutohedron which we call 
a \emph{Bruhat interval polytope}.  
We also prove analogous results for the full symmetric Toda hierarchy,
by mapping our f-KT solutions to those of the full symmetric
Toda hierarchy.
In the Appendix we show that Bruhat interval polytopes 
 are
\emph{generalized permutohedra}, in the sense of Postnikov \cite{Postnikov2},
and that their edges correspond to 
cover relations in the Bruhat order. 
\end{abstract}

\maketitle

\setcounter{tocdepth}{1}
\tableofcontents

\section{Introduction}

The  Toda lattice, introduced by Toda in 1967 (see 
\cite{Toda} for a comprehensive treatment), 
is an integrable Hamiltonian system representing the dynamics of $n$ particles
of unit mass, moving on a line under the influence of exponential 
repulsive forces.  The dynamics can be encoded by 
 a matrix equation called the \emph{Lax equation}
\begin{equation} \label{Lax:1}
\frac{dL}{dt}=[\pi_{\mathfrak{so}}(L),L],
\end{equation}
where $L$ is a tridiagonal symmetric matrix and $\pi_{\mathfrak{so}}(L)$
is the \emph{skew-symmetric projection} of $L$.   More specifically, 
\begin{equation*}\label{tridiagonalL}
L=L(t) = \begin{pmatrix}
b_1 &a_1 & 0 &\cdots& 0\\
a_1& b_2& a_2&\cdots&0\\
0& a_2&b_3&\cdots &0\\
\vdots&\ddots&\ddots&\ddots&\vdots\\
0&\cdots&\cdots&a_{n-1}&b_n
\end{pmatrix}\qquad\text{and}\quad  \pi_{\mathfrak{so}}(L):=(L)_{>0}-(L)_{<0},
\end{equation*}
where $(L)_{>0}$ (respectively $(L)_{<0}$) is the strictly upper 
(resp. lower) triangular part of $L$, so that $\pi_{\mathfrak{so}}(L)$
represents a skew-symmetrization of the matrix $L$.  The entries $a_i$ and $b_j$ 
of $L$ are functions of $t$.

The Toda lattice gives an \emph{iso-spectral deformation}
of the eigenvalue problem of $L$, that is, 
the eigenvalues of $L(t)$ are independent of $t$.
It is an immediate consequence of the Lax equation 
that for any positive integer $k$, 
the trace $\text{tr}(L^k)$ of $L^k$ is a constant of motion
(it is invariant under the Toda flow). These 
invariants  are
the  power sum symmetric functions
of the eigenvalues, and are sometimes referred to as
\emph{Chevalley invariants}.  Note that assuming $\text{tr}(L)=0$
(i.e. $L\in\mathfrak{sl}_n(\R)$),
we have $n-1$ independent Chevalley invariants, $H_k:=\text{tr}(L^{k+1})$
for $k=1,\ldots,n-1$.

One remarkable property of the Toda lattice is that 
for generic initial data, i.e. 
$L$ has
distinct eigenvalues 
\[
\lambda_1~<~\lambda_2~<~\cdots~<~\lambda_n
\]
and $a_k(0)\ne0$ for all $k$, 
the asymptotic form of the Lax matrix is given by
\begin{equation}\label{asymptoticT}
L(t)\quad\longrightarrow\quad\left\{\begin{array}{lll}
\text{diag}(\lambda_1,\lambda_2,\ldots,\lambda_n)\quad&\text{as}\quad t\to-\infty\\[0.5ex]
\text{diag}(\lambda_n,\lambda_{n-1},\ldots,\lambda_1)\quad&\text{as}\quad t\to\infty.
\end{array}\right.
\end{equation}
In other words, all off-diagonal elements $a_i(t)$ approach $0$, 
and the time evolution of the Toda lattice sorts the eigenvalues
of $L$ (see \cite{DNT, Symes}).  This property is referred to as the 
\emph{sorting property}, which has important applications
to matrix eigenvalue algorithms.  
It is  also known
that if we let $L$ range over all tridiagonal matrices with 
fixed eigenvalues $\lambda_1 < \dots < \lambda_n$, 
the set of fixed points of the Toda lattice
 -- i.e. those points $L$ such that 
$dL/dt=0$ -- are precisely the diagonal matrices.  Therefore
 there are $|\Sym_n|=n!$
fixed points of the Toda lattice, where $\Sym_n$ is the symmetric group 
on $n$ letters.

The \emph{full symmetric Toda lattice} is a generalization
of the Toda lattice: it is defined using  
the Lax equation \eqref{Lax:1},  but now 
$L$ can be any symmetric matrix.
For generic $L$, the full symmetric Toda lattice is again
an  integrable Hamiltonian system \cite{DLNT} and 
it has the same asymptotic behavior from \eqref{asymptoticT}  \cite{KM}.
Recently the non-generic flows
were studied in \cite{CSS}, 
including their 
asymptotics as $t \to \pm \infty$.

In this paper we consider a different generalization of 
the Toda lattice called the \emph{full Kostant-Toda lattice}, or 
\emph{f-KT lattice}, first
studied in \cite{EFS}.
Like the Toda lattice and the full symmetric Toda lattice,
the f-KT lattice is an integrable Hamiltonian system,
whose constants of motion are given by the so-called 
{\it chop integrals} \cite{EFS}.\footnote{Note, however, that 
we do not use the chop integrals in our study of the 
f-KT flows.}
It is defined by the Lax equation  
\begin{equation}\label{Lax}
\frac{dL}{dt}=[(L)_{\ge 0},L],
\end{equation}
where $L$ is a \emph{Hessenberg matrix}, i.e. any matrix of the form 
\begin{equation}\label{L}
L=\begin{pmatrix}
a_{1,1} & 1 & 0 & \cdots & 0 \\
a_{2,1}&a_{2,2}&1&\cdots &0\\
\vdots &\vdots &\ddots &\ddots & \vdots\\
a_{n-1,1}&a_{n-1,2}&\cdots&\cdots&1\\
a_{n,1}& a_{n,2}&\cdots&\cdots & a_{n,n}
\end{pmatrix},
\end{equation}
and $(L)_{\geq 0}$ denotes the weakly upper triangular part of $L$.
In terms of the entries $a_{i,j}=a_{i,j}(t)$,  
the f-KT lattice is defined by the system of equations
\begin{equation}\label{eq:f-KTcomp}
\frac{da_{\ell+k,k}}{dt}=a_{\ell+k+1,k}-a_{\ell+k,k-1}+(a_{\ell+k,\ell+k}-a_{k,k})a_{\ell+k,k}
\end{equation}
for  $k=1,\ldots,n-\ell$ and $\ell=0,1,\ldots,n-1$. Here we use the convention
that $a_{i,j}=0$  if 
$j= 0$ or $i= n+1$.  Note  that  the index $\ell$ represents the $\ell$th subdiagonal of the matrix $L$, that is, $\ell=0$ corresponds to 
the diagonal elements
$a_{k,k}$ and $\ell=1$ corresponds to the elements $a_{k+1,k}$ of the 
first subdiagonal, etc.  The f-KT lattice also gives an iso-spectral deformation
of the matrix $L$.

Both the full symmetric Toda lattice and the 
full Kostant-Toda lattice have the same number of free parameters,
namely $\frac{n(n+1)}{2}$.
One should note, however, that when working over $\R$, one can map each
full symmetric matrix to
a Hessenberg matrix, but not vice-versa in general.
For example, the symmetric matrix $\begin{pmatrix}b_1&a\\a&b_2\end{pmatrix}$
can be mapped to the Hessenberg matrix $\begin{pmatrix}
b_1&1\\a^2&b_1\end{pmatrix}$. However,  the 
Hessenberg matrix $\begin{pmatrix}
\beta_1&1\\\alpha&\beta_2\end{pmatrix}$ \emph{cannot} be mapped to a real symmetric matrix if $\alpha<0$.
In this sense, the f-KT lattice may be considered to be more general than the full symmetric Toda lattice.

We also consider the f-KT hierarchy which consists of the symmetries of the f-KT lattice generated by the Chevalley invariants $H_k=\text{tr}(L^{k+1})$ for $k=1,\ldots,n-1$, 
see e.g. \cite{KS08}.
Each symmetry is given by the equation
\begin{equation}\label{KT-hierarchy}
\frac{\partial L}{\partial t_k}=[(L^k)_{\ge0}, L],\qquad\text{for}\quad k=1,2,\ldots, n-1.
\end{equation}
We let $\mathbf{t}:=(t_1,\ldots,t_{n-1})$ denote the multi-time variables
representing the flow parameters.
Note that the flows commute with each other, 
and the equation for $k=1$ corresponds to the f-KT lattice,
with $t=t_1$.

It is well-known that the solution space of the Toda lattice (and its generalizations) 
can be described by the flag variety $G/B^+$ -- see e.g.  \cite{KS08} 
for some basic information on the f-KT lattice and
the full symmetric Toda lattice.  In this paper 
we use as a main tool 
the \emph{totally non-negative part} 
$(G/B^+)_{\geq 0}$ of the flag variety
(which we often abbreviate as the \emph{tnn flag variety}).
As shown by Rietsch \cite{Rietsch} and 
Marsh-Rietsch \cite{MR}, 
$(G/B^+)_{\geq 0}$ has
a decomposition into cells 
$\mathcal R_{v,w}^{>0}$ which are 
indexed by
pairs $(v,w)$ of permutations in $\Sym_n$, where $v \leq w$ in Bruhat order.

Our first main result concerns the asymptotics of f-KT flows associated
to points in the tnn flag variety.
More specifically, 
in Proposition \ref{prop:gsolution}, we associate to each point
$gB^+$ of $(G/B^+)_{\geq 0}$
an initial Hessenberg 
matrix $L^0$ for the f-KT lattice.  The corresponding
f-KT flow is complete for such initial matrices.  In Theorem 
\ref{asympt} 
we prove a generalization of the sorting property: we show that
if $gB^+ \in \mathcal{R}_{v,w}^{>0}$, then as $t \to \pm \infty$,
the diagonal of $L(t)$ contains the eigenvalues, sorted according to
$v$ and $w$, respectively.
Note that we recover the classical sorting property when
$gB^+ \in \mathcal{R}_{e,w_0}^{>0}$, where 
$e=(1,2,\dots,n)$ is the identity permutation, and 
$w_0=(n,n-1,\dots,1)$ is the longest permutation.

Our second main result concerns the moment map images of the
flows of the f-KT hierarchy. 
In Theorem \ref{thm:polytope} we show that
if $gB^+ \in \mathcal{R}_{v,w}^{>0}$, and we apply the moment
map to the corresponding f-KT flow, then its closure is a convex polytope 
$\mathsf{P}_{v,w}$
which generalizes the permutohedron.
More specifically, $\mathsf{P}_{v,w}$  is the convex hull of the permutation
vectors $z$ such that $v \leq z \leq w$.  Note that 
$\mathsf{P}_{v,w}$ is precisely the permutohedron when 
$v=e$ and $w=w_0$.
In the appendix we study combinatorial
properties of these \emph{Bruhat interval polytopes}: 
they are Minkowski 
sums of matroid polytopes and hence
\emph{generalized permutohedra} in the sense of Postnikov \cite{Postnikov2}, and 
their edges correspond to cover relations
in Bruhat order.

Finally, we show that one can map the f-KT flows $L(t)$ studied above
(i.e. those coming from points $gB^+$ of the tnn flag variety) to 
full symmetric Toda flows $\mathcal{L}(t)$.  
This allows us to deduce analogues of our previous results 
for the full symmetric Toda lattice.  More 
specifically,  Theorem \ref{thm:symm} proves that 
if $gB^+ \in \mathcal{R}_{v,w}^{>0}$, 
then as $t \to \pm \infty$, $\mathcal{L}(t)$ tends to a diagonal matrix
with eigenvalues sorted according to $v$ and $w$, respectively.
And Theorem \ref{thm:symmpolytope}
proves that the closure of the image of the moment map 
applied to such a symmetric Toda flow is the Bruhat
interval polytope $\mathsf{P}_{v,w}$.


The structure of this paper is as follows. 
In Sections \ref{sec:flag} and \ref{sec:Grass} we provide 
some background on the flag variety, the Grassmannian,
and their non-negative parts.  In Section \ref{sec:fKT} we 
introduce the full Kostant-Toda lattice (and hierarchy), 
and define the so-called \emph{$\tau$-functions} for the hierarchy.
We then explain how to express the solution $L(t)$ 
using the LU-factorization.  In Section \ref{sec:solKT}
we associate to each point $gB^+$ of the tnn flag variety an initial 
matrix $L^0$, and we analyze the asymptotics 
of the corresponding flow as $t \to \pm \infty$.
In Section \ref{sec:moment} we study the 
moment map images of the f-KT flows associated to points in 
$(G/B^+)_{\geq 0}$.  In 
Section \ref{sec:symmetricToda} we 
map our f-KT solutions (which come from the tnn 
flag variety) to the symmetric Toda lattice, which allows
us to transfer our results on the asymptotics and moment polytopes
to the symmetric Toda lattice.  We end this paper with 
a self-contained appendix which explores combinatorial properties
of Bruhat interval polytopes.

\medskip
\emph{Acknowledgements:}
One of the authors (Y. K.) is grateful to Michael Gekhtman for
useful discussion, in particular, 
his helpful explanation of  \cite[Theorem 3.1]{BG}.


\section{The flag variety and its tnn part}\label{sec:flag}

In this section
we provide some background on the flag variety and  its totally non-negative 
part, the \emph{tnn flag variety}.  In subsequent sections 
we use the geometry of the flag variety (and the Grassmannian)
to describe  combinatorial aspects of the solutions to the full
Kostant-Toda lattice.

\subsection{The flag variety}
Throughout this paper, we consider the group $G=\text{SL}_n(\R)$ and the 
Lie algebra $\mathfrak{g}=\mathfrak{sl}_n(\R)$.
Let $B^+$ and $B^-$ be the Borel subgroups of 
upper and lower-triangular matrices.
Let $U^+$ and $U^-$ be the 
unipotent radicals of 
$B^+$ and $B^-$; these are the
subgroups of upper and lower-triangular matrices with
1's on the diagonals. We let $H$ denote the Cartan subgroup of $G$,
which is the subgroup of diagonal matrices.
We let $\mathfrak{b}^{\pm}, \mathfrak{u}^{\pm}$ and $\mathfrak{h}$
denote their Lie algebras.

For each $1 \leq i \leq n-1$ we have a homomorphism
$\phi_i:{\rm SL}_2\to{\rm SL}_n$ such that
\[
\phi_i\begin{pmatrix} a& b\\c&d\end{pmatrix}=
\begin{pmatrix}
1 &             &       &       &           &     \\
   &\ddots  &        &      &            &        \\
   &             &   a   &   b  &          &       \\
   &             &   c    &   d  &         &       \\
   &            &          &       &  \ddots  &     \\
   &            &          &       &              & 1
   \end{pmatrix} ~\in~ {\rm SL}_n,
\]
that is, $\phi_i$ replaces a $2\times 2$ block of the identity matrix with $\begin{pmatrix} a&b\\c&d\end{pmatrix}$, where $a$ is at the $(i,i)$-entry.
We have $1$-parameter subgroups of $G$
defined by
\begin{equation*}
x_i(m) = \phi_i \left(
                   \begin{array}{cc}
                     1 & m \\ 0 & 1\\
                   \end{array} \right)  \text{ and }\ 
y_i(m) = \phi_i \left(
                   \begin{array}{cc}
                     1 & 0 \\ m & 1\\
                   \end{array} \right) ,\
\text{ where }m \in \R.
\end{equation*}

Let $W$ denote the Weyl group $ N_G(T) / T$,
where $N_G(T)$ is the normalizer of the maximal torus $T$.
The simple reflections $s_i \in W$ are given by
$s_i:= \dot{s_i} T$ where $\dot{s_i} :=
                 \phi_i \left(
                   \begin{array}{cc}
                     0 & -1 \\ 1 & 0\\
                   \end{array} \right)$,
and any $w \in W$ can be expressed as a product $w = s_{i_1} s_{i_2}
\dots s_{i_\ell}$ with $\ell=\ell(w)$ factors.  We set $\dot{w} =
\dot{s}_{i_1} \dot{s}_{i_2} \dots \dot{s}_{i_\ell}$.
In our setting $W$ is isomorphic to
$\Sym_n$, the symmetric group on $n$ letters,
and $s_i$ corresponds to the transposition exchanging $i$ and $i+1$.
We let $\leq$ denote the (strong) Bruhat order on $W$.

\begin{definition}
The \emph{real flag variety} is the variety of all \emph{flags}
\[
\{V_{\bullet} = V_1~\subset~ V_2~\subset~\cdots~\subset~V_n=\R^n \ \vert \ \dim V_i = i\}
\]
of vector subspaces of $\R^n$.  As we explain below,
it can be identified with the homogeneous space
$G/B^+$.
\end{definition}

Let $\{e_1,\dots, e_n\}$ be the standard basis of $\R^n$.
The \emph{standard flag} is 
$E_{\bullet} = E_1~\subset~ E_2~\subset~\cdots~\subset~E_n$ where
$E_i$ is the span of $\{e_1, \dots, e_i\}$.  Note that the group $G$ acts on 
flags; the Borel subgroup $B^+$ is the stabilizer of the standard flag
$E_{\bullet}$.   Any flag $V_{\bullet}$ may be written as 
$g E_{\bullet}$ for some $g \in G$.  Note that 
$g E_{\bullet} = g' E_{\bullet}$ if and only if $gB^+ = g'B^+.$
In this way we may identify the flag variety with 
$G/B^+$.


We have two opposite Bruhat
decompositions of $G/B^+$:
\begin{equation*}
G/B^+=\bigsqcup_{w\in W} B^+ \dot w B^+/B^+=\bigsqcup_{v\in W}
B^- \dot v B^+/B^+.
\end{equation*}
We define the intersection of opposite Bruhat cells
\begin{equation*}
\mathcal R_{v,w}:=(B^+\dot w B^+/B^+)\cap (B^-\dot v B^+/B^+),
\end{equation*}
which is nonempty
precisely when $v\le w$.
The strata $\mathcal R_{v,w}$ are often called \emph{Richardson varieties}.

\subsection{The tnn part of the flag variety}
\begin{definition} \cite{Lusztig3}
The \emph{tnn
part $U_{\geq 0}^-$ of $U^-$} is defined to be the semigroup in
$U^-$ generated by the $y_i(p)$ for $p \in \R_{\geq 0}$.
The
\emph{tnn part 
$(G/B^+)_{\geq 0}$
of $G/B^+$} 
is defined by \begin{equation*}
(G/B^+)_{\geq 0} := \overline{ \{\,u  B^+ ~|~ u \in U_{\geq 0}^-\, \} },
\end{equation*}
where the closure is taken inside $G/B^+$ in its real topology.
We sometimes refer to $(G/B^+)_{\geq 0}$ as the 
\emph{tnn flag variety}.
\end{definition}



Lusztig \cite{Lusztig3, Lusztig2} introduced a 
natural decomposition of $(G/B^+)_{\geq 0}$.

\begin{definition}  \cite{Lusztig3}
For $v, w \in W$ with $v \leq w$, let
\begin{equation*}
\mathcal R_{v,w}^{>0} := \mathcal R_{v,w} \cap (G/B^+)_{\geq 0}.
\end{equation*}
Then the tnn part of the flag variety $G/B^+$ has the decomposition,
\begin{equation}\label{decomposition}
(G/B^+)_{\ge 0}=\bigsqcup_{w\in \Sym_n}\left(\bigsqcup_{v\le w}\mathcal{R}^{>0}_{v,w}\right).
\end{equation}
\end{definition}

Lusztig conjectured and Rietsch proved \cite{Rietsch} 
that $\mathcal R_{v,w}^{>0}$ 
is a semi-algebraic cell of dimension
$\ell(w)-\ell(v)$.  
Subsequently Marsh-Rietsch \cite{MR}  provided an explicit parameterization of each 
cell.  To state their result, we first review the notion of 
positive distinguished subexpression, as in 
\cite{Deodhar} and \cite{MR}.

Let $\w:= s_{i_1}\dots s_{i_m}$ be a reduced expression for $w\in W$.
A  {\it subexpression} $\v$ of $\w$
is a word obtained from the reduced expression $\w$ by replacing some of
the factors with $1$. For example, consider a reduced expression in $\Sym_4$, say $s_3
s_2 s_1 s_3 s_2 s_3$.  Then $1\, s_2\, 1\, 1\, s_2\, s_3$ is a
subexpression of $s_3 s_2 s_1 s_3 s_2 s_3$.
Given a subexpression $\v$,
we set $v_{(k)}$ to be the product of the leftmost $k$
factors of $\v$, if $k \geq 1$, and $v_{(0)}=1$.

\begin{definition}\label{d:Js}\cite{Deodhar, MR}
Given a subexpression $\v$ of  $\w=
s_{i_1} s_{i_2} \dots s_{i_m}$, we define
\begin{align*}
J^{\circ}_\v &:=\{k\in\{1,\dotsc,m\}\ |\  v_{(k-1)}<v_{(k)}\},\\
J^{+}_\v\, &:=\{k\in\{1,\dotsc,m\}\ |\  v_{(k-1)}=v_{(k)}\},\\
J^{\bullet}_\v &:=\{k\in\{1,\dotsc,m\}\ |\  v_{(k-1)}>v_{(k)}\}.
\end{align*}

The subexpression  $\v$
is called {\it
non-decreasing} if $v_{(j-1)}\le v_{(j)}$ for all $j=1,\dotsc, m$,
e.g.\ if $J^{\bullet}_\v=\emptyset$.
It is called {\it distinguished}
if we have
$v_{(j)}\le v_{(j-1)}\hspace{2pt} s_{i_j}$ for all
$j\in\{1,\dotsc,m\}.$
In other words, if right multiplication by $s_{i_j}$ decreases the
length of $v_{(j-1)}$, then in a distinguished subexpression we
must have
$v_{(j)}=v_{(j-1)}s_{i_j}$.
Finally, 
 $\v$ is called a {\it positive distinguished subexpression}
(or a PDS for short) if
$v_{(j-1)}< v_{(j-1)}s_{i_j}$ for all
$j\in\{1,\dotsc,m\}$.
In other words, it is distinguished and non-decreasing.
\end{definition}

\begin{lemma}\label{l:positive}\cite{MR}
Given $v\le w$
and a reduced expression $\w$ for $w$,
there is a unique PDS $\v_+$ for $v$ contained in $\w$.
\end{lemma}


\begin{theorem}\label{t:parameterization}\cite[Proposition 5.2, Theorem 11.3]{MR}
Choose a reduced expression $\w=s_{i_1} \dots s_{i_m}$ for $w$ with $\ell(w)=m$.
To $v \leq w$ we associate the unique 
PDS 
$\v_+$ for $v$ in $\w$.  Then $J^{\bullet}_{\v^+} = \emptyset$.
We define 
\begin{equation}\label{eq:G}
G_{\v_+,\w}^{>0}:=\left\{g= g_1 g_2\cdots g_m \left
|\begin{array}{ll}
 g_\ell= y_{i_\ell}(p_\ell)& \text{ if $\ell\in J^{+}_\v$,}\\
 g_\ell=\dot s_{i_\ell}& \text{ if $\ell\in J^{\circ}_\v$,}
 \end{array}\right. \right\},
\end{equation}
where each $p_\ell$ ranges over $\R_{>0}$.
The set $G_{\v_+,\w}^{>0}$ lies in $U^-\dot{v}\cap B^+\dot{w}B^+$,
$G_{\v_+,\w}^{>0} \cong \R_{>0}^{\ell(w)-\ell(v)}$,
and the map $g\mapsto g B^+$ defines an isomorphism 
\begin{align*}\label{e:parameterization+}
G_{\v_+,\w}^{>0}&~\overset\sim\To ~\mathcal R_{v,w}^{>0}.
\end{align*}
\end{theorem}

\begin{example}\label{ex:param}
Consider the reduced decomposition
$\w =  s_2 s_3 s_1 s_4 s_5 s_3 s_2 $ for $w \in W=S_6$.
Let $v = s_3 s_4 s_2 \leq w$.  Then 
the PDS $\v_+$ for $v$ in $\w$ is 
$\v_+ = 1 s_3 1 s_4 1 1 s_2$.  The set
$G_{\v_+,\w}^{>0}$ consists of all elements of the form 
\begin{equation*}\label{flagmatrix}
y_2(p_1) \dot s_3 y_1(p_3) \dot s_4 y_5(p_5) y_3(p_6) \dot s_2=
\begin{pmatrix} 
1& 0& 0 & 0 & 0 & 0\\
p_3 & 0 & -1 & 0 & 0 & 0\\
p_1 p_3 & 0 & -p_1 & 0 & 1 & 0\\
0 & 1 & 0 & 0 & 0 & 0\\
0 & p_6 & 0 & 1 & 0 & 0\\
0 & 0 & 0 & 0 & p_5 & 1
\end{pmatrix} \quad\text{ where each }p_i \in \R_{>0}.
\end{equation*}
\end{example}

\section{The Grassmannian and its tnn part}\label{sec:Grass}

\subsection{The Grassmannian}

The \emph{real Grassmannian} $Gr_{k,n}$ is the space of all
$k$-dimensional subspaces of $\R^n$.  An element of
$Gr_{k,n}$ can be viewed as a full-rank $k\times n$ matrix $A$ modulo left
multiplication by nonsingular $k\times k$ matrices.  In other words, two
$k\times n$ matrices are equivalent, i.e. they 
represent the same point in $Gr_{k,n}$, if and only if they
can be obtained from each other by row operations.  

Let $\binom{[n]}{k}$ be the set of all $k$-element subsets of $[n]:=\{1,\dots,n\}$.
For $I\in \binom{[n]}{k}$, let $\Delta_I(A)$
be the {\it Pl\"ucker coordinate}, that is, the maximal minor of the $k\times n$ matrix $A$ located in the column set $I$.
The map $A\mapsto (\Delta_I(A))$, where $I$ ranges over $\binom{[n]}{k}$,
induces the {\it Pl\"ucker embedding\/} $Gr_{k,n}\hookrightarrow \mathbb{RP}^{\binom{n}{k}-1}$.

Just as for the flag variety, one may identify the
Grassmannian with a homogeneous space. 
Let $P_k$ be the parabolic subgroup which fixes the 
$k$-dimensional subspace spanned by $e_1,\dots, e_k$.
(This is a block upper-triangular matrix containing $B^+$.)
Then we may identify $Gr_{k,n}$ with the 
space of cosets $G/P_k$.



There is a natural projection 
$\pi_k:G/B^+ \to Gr_{k,n}$ such that 
$\pi_k(V_1 \subset \dots \subset V_n)=V_k$.
One may equivalently express this projection as the map
$\pi_k:G/B^+ \to G/P_k$, where 
$\pi_k(gB^+) = gP_k$.
Abusing notation, 
we simply write $\pi_k(g)=A_k$   with $A_k\in Gr_{k,n}\cong G/P_k$
instead of  $\pi_k(gB^+)=gP_k$.

Concretely, for $g\in G$, $\pi_k(g)$ is represented 
by the $k\times n$ matrix
$A_k$ consisting of the leftmost $k$ columns
of $g$, i.e. 
\begin{equation}\label{eq:pi}
g=\begin{pmatrix}
g_{1,1}&\cdots&g_{1,k}&\cdots &g_{1,n}\\
\vdots&\ddots&\vdots&\vdots&\vdots\\
g_{k,1}&\cdots&g_{k,k}&\cdots& g_{k,n}\\
\vdots&\vdots&\vdots&\vdots&\vdots\\
g_{n,1}&\cdots&g_{n,k}&\cdots& g_{n,n}
\end{pmatrix}\quad\longmapsto
\quad A_k=\begin{pmatrix}
g_{1,1}&\cdots&g_{k,1}&\cdots&g_{n,1}\\
\vdots&\ddots & \vdots&\vdots&\vdots\\
g_{1,k}&\cdots&g_{k,k}&\cdots&g_{n,k}
\end{pmatrix}.
\end{equation}
This is equivalent to the following formula using the Pl\"ucker embedding
into the projectivization of the wedge product space $\mathbb{P}(\wedge^k\R^n)
\cong\mathbb{RP}^{\binom{n}{k}-1}$ with the standard basis
$\{e_i:i=1,\ldots,n\}$,
\begin{equation}\label{standard}
g\cdot e_1\wedge\cdots\wedge e_k=\sum_{1\le i_1<\cdots<i_k\le n}\Delta_{i_1,\ldots,i_k}(A_k)\,e_{i_1}\wedge\cdots \wedge e_{i_k}.
\end{equation}
The Pl\"ucker coordinates $\Delta_{i_1,\ldots,i_k}(A_k)$ are then given by
\[
\Delta_{i_1,\ldots,i_k}(A_k)=\langle e_{i_1}\wedge\cdots\wedge e_{i_k},~g\cdot e_1\wedge\cdots\wedge e_k\rangle,
\]
where $\langle\cdot,\cdot\rangle$
is the usual inner product on $\wedge^k\R^n$. 
\begin{remark}\label{rem:basis}
There is a variant of the projection $\pi_k$ that will be useful to us later.
Let $\{u_i:i=1,\ldots,n\}$ be an ordered basis of $\R^n$.  Then
one can define the map $\tilde\pi_k(g)=\tilde A_k$ by
\begin{equation}\label{basis}
g\cdot e_1\wedge\cdots\wedge e_k=\sum_{1\le i_1<\cdots <i_k\le n}\Delta_{i_1,\ldots,i_k}(\tilde A_k)\,u_{i_1}\wedge\cdots \wedge u_{i_k}.
\end{equation}
Using the relation $u_{i_1}\wedge\cdots\wedge u_{i_k}=U\cdot e_{i_1}\wedge\cdots\wedge e_{i_k}$ with  the $n\times n$ matrix $U=(u_1,\ldots,u_n)$, 
 the 
Pl\"ucker coordinates $\Delta_{i_1,\ldots,i_k}(\tilde A_k)$  are then given by
\begin{align*}
\Delta_{i_1,\ldots,i_k}(\tilde A_k)
&=\langle e_{i_1}\wedge\cdots\wedge e_{i_k},~
U^{-1} g\cdot e_1\wedge\cdots\wedge e_k\rangle.
\end{align*}
This implies that  
$\tilde \pi_k(g)= \tilde A_k = \pi_k(U^{-1} g)$ and hence in general,
$A_k$ and $\tilde A_k$ represent different element 
of $Gr_{k,n}$.  
The map $\tilde\pi_k$ will be useful in Section \ref{sec:moment} where we consider the  moment map (see Lemma \ref{lem:momentprojection}).
\end{remark}

\subsection{The tnn part of the Grassmannian}

One may define the tnn part of the Grassmannian
$(Gr_{k,n})_{\geq 0}$  either as the projection
of $(G/B^+)_{\geq 0}$, following Lusztig \cite{Lusztig2}, 
or equivalently in terms of Pl\"ucker coordinates,
following Postnikov \cite{Postnikov}. 

\begin{definition}\label{def:TNNGrass}
The \emph{tnn
part of the Grassmannian} $(Gr_{k,n})_{\geq 0}$
is the image $\pi_k ((G/B^+)_{\geq 0})$.
Equivalently, 
$(Gr_{k,n})_{\geq 0}$ is the subset of $Gr_{k,n}$ such that all
Pl\"ucker coordinates are non-negative.
\end{definition}

Let $W_k = \langle s_1,\dots,\hat{s}_{k},\dots,s_{n-1} \rangle$
be a parabolic subgroup of $W = \Sym_n$.
Let $W^k$ denote the set of minimal-length
coset representatives of $W/W_k$.
Recall that a \emph{descent} of a permutation $z$
is a position $j$ such that $z(j)>z(j+1)$.
Then $W^k$ is the subset of permutations of $\Sym_n$
which have at most one descent, and if it exists,
that descent must be in position $k$.

\begin{definition}
For each $w\in W^k$ and $v \leq w$,  define
$\mathcal P_{v,w}^{>0} = \pi_k(\mathcal R_{v,w}^{>0})$.
\end{definition}

\begin{theorem}\cite{Rietsch}\label{Grasscell}
We have a decomposition
\begin{equation*}\label{eq:GrassCell} 
(Gr_{k,n})_{\geq 0}  
 = \bigsqcup_{w\in W^k} \bigsqcup_{v \leq w} \mathcal P_{v,w}^{>0}
\end{equation*}
of $(Gr_{k,n})_{\geq 0}$ into cells.
\end{theorem}

\begin{remark}
If $w \in W^k$, then 
the projection $\pi_k:G/B^+ \to Gr_{k,n}$ is an isomorphism when 
restricted to $\mathcal R_{v,w}$, and hence
is an isomorphism 
from $\mathcal R_{v,w}^{>0}$ to the corresponding cell
$\mathcal P_{v,w}^{>0}$ of $(Gr_{k,n})_{\geq 0}$.
More generally, for $w$ not necessarily in $W^k$, 
$\pi_k$ is a surjective map
taking cells of $(G/B^+)_{\geq 0}$ to cells of $(Gr_{k,n})_{\geq 0}$.
\end{remark}

There is another description of the cell decomposition from 
Theorem \ref{Grasscell} which is due to Postnikov \cite{Postnikov}, 
who discovered it independently.  To give Postnikov's description,
we first review the matroid stratification of the Grassmannian.

\begin{definition}
Given $A \in Gr_{k,n}$, let $\mathcal M(A)$ be the collection 
of $k$-element subsets $I$ of $[n]$ such that $\Delta_I(A) \neq 0$.
Such a collection is the set of \emph{bases} of a \emph{(realizable) matroid}.
Now given a collection $\mathcal M$ of $k$-element subsets of $[n]$,
we define the \emph{matroid stratum}
$S_{\mathcal{M}}$ to be
$$S_{\mathcal{M}} = \{A \in Gr_{k,n} \ \vert \ \mathcal M(A) = \mathcal{M} \}.$$
\end{definition}

Note that $S_{\mathcal{M}}$ may be empty.
Letting $\mathcal{M}$ vary over all subsets of ${[n] \choose k}$, 
we have the \emph{matroid stratification} of the Grassmannian
$Gr_{k,n} = \sqcup_{\mathcal{M}} S_{\mathcal{M}}.$

\begin{theorem}\cite{Postnikov}
The intersection of each matroid stratum $S_{\mathcal{M}}$
with $(Gr_{k,n})_{\geq 0}$ is either empty or is a cell.
This gives a cell decomposition of 
$(Gr_{k,n})_{\geq 0}$.
\end{theorem}

\subsection{Flag minors of the tnn flag variety} \label{sec:project}

\begin{definition}
Let $M$ be an $n \times n$ matrix with real entries.  
Any $k \times k$ submatrix (for $1 \leq k \leq n$) is called
a \emph{flag minor} if its set of columns is precisely $\{1,2,\dots,k\}$, 
the leftmost $k$ columns of $M$.  
And we say that $M$ is \emph{flag  non-negative} 
if all of its flag minors
are non-negative.  
\end{definition}

\begin{remark}
Note that the flag minors of $g \in G$  are precisely the Pl\"ucker 
coordinates of the projections of $gB^+$
to the various Grassmannians $\pi_k(gB^+)$
for $1 \leq k\leq n$.
\end{remark}

\begin{lemma}\label{lem:flagTNN}
Let $v \leq w$ be elements of $W$.  Choose a reduced expression
$\w$ for $w$ and let $\v_+$ be the PDS for $v$ within $\w$.  
Then $G_{\v_+,\w}^{>0}$ consists of flag non-negative matrices, i.e.
for any $g\in G_{\v_+,\w}^{>0}$, all flag minors of $g$
are non-negative.
\end{lemma}

\begin{proof}
Choose any $k$ such that $1 \leq k \leq n$.
Recall from Definition \ref{def:TNNGrass} that 
the projection $\pi_k$ maps the tnn flag
variety $(G/B^+)_{\geq 0}$ to the tnn Grassmannian
$(Gr_{k,n})_{\geq 0}$, i.e. the subset of the real Grassmannian such
that all $k \times k$ minors are non-negative.  When we identify
$G_{\v_+,\w}^{>0}$ with the cell 
$\mathcal R_{v,w}^{>0}$, this projection maps
a matrix $g \in G_{\v_+,\w}^{>0}$ to the span of 
the leftmost $k$ columns of $g$.
It follows that the $k \times k$ flag minors of $g$ are non-negative.
\end{proof}

We now give a sufficient condition for 
a flag minor of 
$G_{\v_+,\w}^{>0}$ to be positive.

Given an element $g\in G_{\v_+,\w}^{>0}$ and a $k$-element set 
$I_k=\{i_1<i_2<\cdots<i_k\}\in\binom{[n]}{k}$, we 
let $\Delta_{I_k}^k(g)$ denote the minor of the matrix $g$ given by
\begin{equation*}\label{g-minor}
\Delta_{I_k}^k(g):=\langle e_{i_1}\wedge\cdots\wedge e_{i_k},\,g\cdot e_1\wedge\cdots\wedge e_k\rangle.
\end{equation*}
In other words, $\Delta_{I_k}^k(g)=\Delta_{I_k}(A_k)$ is 
the \emph{flag minor} of $g$ which uses 
the leftmost $k$ columns of $g$ and the rows indexed by  $I_k$.

Let $[k]:=\{1,\ldots,k\}$.  For any $z\in W$ we define the ordered set
$z \cdot [k] = \{z(1), \dots, z(k)\}$.  Then we have the following.  
 
\begin{lemma}\label{minors}
Let $v \leq w$ be elements in $W=\Sym_n$, and choose
$z\in \Sym_n$ arbitrarily.  Choose a reduced subexpression
$\w$ for $w$; this determines the PDS
$\v_+$ for $v$ in $\w$.
Choose  any $g\in G_{\v_+,\w}^{>0}$.  Then we have
\[
\Delta_{z\cdot [k]}^k(g) > 0\qquad  \text{ for}\quad  1\leq k \leq n  
\]
if and only if 
$$v \leq z \leq w.$$
\end{lemma}

The classical \emph{tableau criterion} for Bruhat order on $\Sym_n$
will be useful for the proof.
\begin{lemma} \label{tableaucriterion}\cite{Ehresmann}
We have that $x \leq y$ in Bruhat order on $\Sym_n$ if and only if
$x_{i,k} \leq y_{i,k}$ for all $1 \leq i \leq k \leq n-1$,
where $x_{i,k}$ is the $i$th entry in the increasing rearrangement
of $x_1, x_2, \dots, x_k$, and similarly for $y_{i,k}$.
\end{lemma}

We now prove Lemma \ref{minors}.
\begin{proof}
First suppose that $z$ satisfies $v \leq z \leq w$.
Then by \cite{RietschClosure}, 
the cell $\RR_{z,w}^{>0}$ is contained in the closure
$\overline{\RR_{v,w}^{>0}}$, and hence
$\pi_k(\RR_{z,w}^{>0}) \subset \pi_k (\overline{\RR_{v,w}^{>0}})$.
Recall from Section \ref{sec:Grass} that the projection $\pi_k$ of any cell
of $(G/B^+)_{\geq 0}$ is a cell of $(Gr_{k,n})_{\geq 0}$, 
and that cells of $(Gr_{k,n})_{\geq 0}$ are defined by specifying 
that certain Pl\"ucker coordinates are $0$ and all others are 
strictly positive.  Therefore if 
$\mathcal C_1$ and $\mathcal C_2$ are two cells of $(Gr_{k,n})_{\geq 0}$
such that $\mathcal C_1 \subset \overline{\mathcal C_2}$,
and the Pl\"ucker coordinate $\Delta_J$ is positive 
on all points of $C_1$, then it must also be positive on 
all points of $\mathcal C_2$.

Recall that $\RR_{z,w}^{>0} \subset \RR_{z,w} = (B^+ \dot w  B^+/B^+)
\cap (B^- \dot z  B^+/B^+)$.
Then $\pi_k(\RR_{z,w}^{>0}) \subset \pi_k(B^- \dot z B^+/B^+)$.
It is well known that the lexicographically minimal nonvanishing
Pl\"ucker coordinate on $\pi_k(B^-\dot{z} B^+/B^+)$
 is $\Delta_{z\cdot[k]}$.  Therefore 
$\Delta_{z\cdot[k]}$ must be strictly positive on any point of 
$\pi_k(\RR_{z,w}^{>0})$.

And now since
$\pi_k(\RR_{z,w}^{>0}) \subset 
 \pi_k (\overline{\RR_{v,w}^{>0}})$, 
it follows that the Pl\"ucker coordinate
$\Delta_{z\cdot[k]}$ must be positive on any point 
in $\pi_k(\RR_{v,w}^{>0})$.
Finally recall that the projection $\pi_k$ from $G/B^+$ to 
$Gr_{k,n}$ maps an element 
$g \in G_{\v_+,\w}^{>0}$ to the span of its leftmost $k$ columns,
i.e. $\pi_k(g)=A_k$.
Therefore we must have that 
$\Delta_{z\cdot[k]}(A_k)=\Delta^k_{z \cdot [k]}(g)>0$.

Now suppose that 
$\Delta_{z\cdot [k]}^k(g) > 0$ for all $k$ such that $ 1\leq k \leq n$.
As before, $\pi_k(g)$ represents an element of 
$\pi_k(\mathcal R_{v,w})$, and the lexicographically minimal 
nonvanishing Pl\"ucker coordinate on $\pi_k(\mathcal R_{v,w})$
is $\Delta_{v \cdot [k]}(A_k)$.  Therefore we must have that 
the subset $z \cdot [k]$ is lexicographically greater than 
$v \cdot [k]$.  But since this is true for all $k$,  
Lemma \ref{tableaucriterion} implies that
$v \leq z \leq w$.
\end{proof}

\section{The full Kostant-Toda lattice}\label{sec:fKT}

The $\mathfrak{sl}_n(\mathbb{R})$-\emph{full Kostant-Toda (f-KT) lattice} is defined as follows.
Let $L\in\mathfrak{sl}_n(\R)$ be given by
\[
L\,\in\, \epsilon+\mathfrak{b}^- ,
\]
where $\epsilon$ is the matrix with $1$'s on the superdiagonal and zeros elsewhere.
We use the coordinates \eqref{L} for $L$ and refer to it as a \emph{Lax matrix}.
Then the f-KT lattice is defined by the \emph{Lax equation}
\eqref{Lax}.

One of the goals  of this paper is to describe the behavior of the solution
$L(t)$ of the f-KT lattice when the initial matrix $L(0)=L^0$ is associated to an arbitrary point on the tnn part 
$(G/B^+)_{\geq 0}$ of the flag variety $G/B^+$.
In this section we will give background on the f-KT lattice, including
a characterization of the fixed points, and the method of finding 
solutions via the LU-factorization.

\subsection{The fixed points of the f-KT lattice}
Let  $\mathcal{F}_{\Lambda}$ be the \emph{isospectral variety} consisting
of 
the Hessenburg matrices of \eqref{L} with fixed  eigenvalues
$\Lambda=\{\lambda_1,\lambda_2,\ldots,\lambda_n\}$,\footnote{Note that
we will also occasionally use
$\Lambda$ to denote the matrix
$\text{diag}(\lambda_1,\lambda_2,\ldots,\lambda_n)$.} i.e.
\[
\mathcal{F}_{\Lambda}:=\{L\in\epsilon+\mathfrak{b}^-~|~ L \text{ has eigenvalues }\Lambda\}.
\]
Throughout this paper, we assume that all the eigenvalues are real and distinct, and have the ordering
\[
\lambda_1~<~\lambda_2~<~\cdots~<~\lambda_n.
\]
Note that since $L \in \mathfrak{sl}_n(\mathbb{R})$,  we have
$\sum_{i=1}^n \lambda_i = 0$.

We let
$E$ denote the Vandermonde matrix in the $\lambda_i$'s:
\begin{equation}\label{E}
E:=\begin{pmatrix} 
1& 1& \cdots & 1\\
\lambda_1&\lambda_2&\cdots&\lambda_n\\
\vdots&\vdots&\ddots&\vdots\\
\lambda_1^{n-1}&\lambda_2^{n-1}&\cdots&\lambda_n^{n-1}
\end{pmatrix}.
\end{equation}

\begin{definition}
A Lax matrix $L$ is a \emph{fixed point} 
or \emph{stationary point} of the f-KT lattice
if \[
\frac{dL}{dt}=0,\qquad\text{or equivalently,}\qquad [(L)_{\ge 0},L]=0.
\]
\end{definition}

\begin{lemma}\label{lem:commutation}
Let $L$ be an $n \times n$ matrix with distinct eigenvalues
$\{\lambda_1,\dots,\lambda_n\}$
and let $F$ be an $n\times n$ matrix.  Then
\[
[F,L]=0\qquad\text{implies}\qquad F=\sum_{k=0}^{n-1}c_kL^k
\]
for some constants $c_k$, i.e. $F$ is a polynomial of $L$.
\end{lemma}
\begin{proof}
Since $L$ is diagonalizable,  we can write 
$L = M\Lambda M^{-1}$ for some invertible matrix $M$.  Then
$[F,L]=0$ implies that $F = MDM^{-1}$ for some diagonal matrix $D$.
Let $D=\text{diag}(\mu_1,\ldots,\mu_n)$;
the $\mu_j$'s are the eigenvalues of $F$.  
Then each $\mu_j$ can be expressed as
\[
\mu_j=\sum_{i=0}^{n-1}a_i\lambda_j^i \qquad\text{for}\quad j=1,\ldots,n.
\]
Note that this equation is $\mu=a\,E$ with $\mu:=(\mu_1,\ldots,\mu_n)$
and $a:=(a_0,\ldots,a_{n-1})$.  Since the Vandermonde matrix $E$ is invertible, 
$a$ is uniquely determined for given $\mu$.  We can equivalently 
write this equation as
$D=\sum_{i=1}^{n-1}a_i\Lambda^i$ with $\Lambda=\text{diag}(\lambda_1,\ldots,\lambda_n)$,  and hence we have
\[
FM=MD=M\left(\sum_{i=0}^{n-1}a_i\Lambda^i\right)=\left(\sum_{i=0}^{n-1}a_iL^i\right)M,
\]
where we have used $L=M\Lambda M^{-1}$.  Since $M$ is invertible, this 
implies that  $F=\sum_{i=0}^{n-1}a_iL^i$.
\end{proof}

Lemma \ref{lemma:fixed} allows us to characterize the
fixed points of the f-KT lattice.
\begin{lemma}\label{lemma:fixed}
If $L$ has distinct eigenvalues and 
$[(L)_{\ge0},~L]=0$, then $L=(L)_{\ge0},$
that is, $(L)_{<0}=0$.
\end{lemma}
\begin{proof}
 The decomposition $L=(L)_{\ge 0}+(L)_{<0}$ gives
 \[
 [(L)_{\ge 0},\,L]=-[(L)_{<0},\,L]=0.
 \]
  Then Lemma \ref{lem:commutation} implies that we have
  \[
  (L)_{<0}=\sum_{k=0}^{n-1}c_kL^k.
  \]
Diagonalizing both sides, 
we have the row vector equation $0=c\, E$,
where $c=(c_0,\ldots,c_{n-1})$ and
the Vandermonde matrix $E$ is given by \eqref{E}.  Since $E$ is invertible, we have $c=0$. 
\end{proof}

Then we have the following result.
\begin{proposition}\label{prop:fixedpoint}
Let $L$ be in $\mathcal{F}_{\Lambda}$.  Then
$L$ is a fixed point of the f-KT lattice if and only if
\[
L=(L)_{\ge 0}.
\]
Moreover, the diagonal part of a fixed point $L$ is given by
\[
\text{diag}(L)=\text{diag}(\lambda_{\pi(1)},\lambda_{\pi(2)},\ldots,\lambda_{\pi(n)})
\qquad\text{for some}\quad \pi\in\Sym_n.
\]
\end{proposition}
\begin{proof}
The fact that a matrix of the form 
$L = (L)_{\geq 0}$ is a fixed point of the f-KT lattice follows directly
from the definitions.  The other direction 
is just a corollary of Lemma \ref{lemma:fixed}.

Now suppose $L$ is a fixed point
and hence $L = (L)_{\geq 0}$.  
Since $L$ has the eigenvalues $\Lambda$, 
the diagonal part of $L$
consists of
the eigenvalues.  Any arrangement of the eigenvalues can be
expressed as $(\lambda_{\pi(1)},\ldots,\lambda_{\pi(n)})$ 
with some $\pi\in\Sym_n$.
\end{proof}

\subsection{Matrix factorization and the $\tau$-functions of the f-KT lattice}
To find the solution $L(t)$ of the f-KT lattice in terms of the initial matrix $L(0)=L^0$,
it is standard to consider the \emph{LU-factorization} of the matrix $\exp(tL^0)$:
\begin{equation}\label{exp}
\exp(tL^0)=u(t)b(t)\qquad \text{with}\quad u(t)\in U^-,~ b(t)\in B^+. 
\end{equation}
Note here that $u(0)=b(0)=I$, the identity matrix.
It is known and easy to show that this factorization exists if and only if the principal minors are all nonzero.  We then let $\tau_k(t)$ denote the $k$-th principal minor, i.e.
\begin{equation}\label{eq:tau1}
\tau_k(t):=[\exp(tL^0)]_k\qquad\text{for}\quad k=1,2,\ldots,n.
\end{equation}
These functions, which are called  \emph{$\tau$-functions}, 
play a key role in the method for solving the f-KT lattice, as
we will explain below.  
We first have the following result (see e.g. \cite{Symes} which deals with the
original Toda lattice).
\begin{proposition}\label{prop:LUsolution}
The solution $L(t)$ is given by
\[
L(t)=u^{-1}(t)L^0u(t)=b(t)L^0b(t)^{-1}.
\]
\end{proposition}
\begin{proof}
Taking the derivative of \eqref{exp}, we have
\[
\frac{d}{dt}\exp(tL^0)=L^0ub=ubL^0=\dot{u}b+u\dot{b},
\]
where $\dot{x}$ means the derivative of $x(t)$.  This 
equation can  also be written as
\[
u^{-1}L^0u=bL^0b^{-1}=u^{-1}\dot{u}+\dot{b}b^{-1}.
\]
We denote this as $\tilde{L}$, and show $\tilde{L}=L$.  
Using $\mathfrak{g}=\mathfrak{u}^-\oplus\mathfrak{b}^+$, we decompose $\tilde{L}$ 
as
\[
u^{-1}\dot{u}=(\tilde{L})_{<0}
\qquad\text{and}\qquad
\dot{b}b^{-1}=(\tilde{L})_{\ge 0}.
\]
To show $\tilde L=L$, we first show that $\tilde{L}$ also satisfies the f-KT lattice.  Differentiating $\tilde{L}=u^{-1}L^0u$, we have
\[
\frac{d\tilde{L}}{dt}=-u^{-1}\dot{u}\tilde{L}+\tilde{L}u^{-1}\dot{u}=[-u^{-1}\dot{u},\tilde{L}].
\]
Here we have used $\frac{d}{dt}u^{-1}=-u^{-1}\dot{u}u^{-1}$.
Writing $u^{-1}\dot{u}=(\tilde L)_{<0}=\tilde L-(\tilde L)_{\ge0}$,
we obtain
\[
\frac{d\tilde L}{dt}=[(\tilde L)_{\ge0},\tilde L].
\]
Since $\tilde L(0)=L^0=L(0)$, i.e. the initial data are the same,  the uniqueness theorem of the differential equation implies that $\tilde L(t)=L(t)$.  This completes the proof.
\end{proof}

We also have an explicit formula for the diagonal elements of $L(t)$.
\begin{proposition}\label{prop:diagonal}
The diagonal elements of the matrix $L=L(t)$ can be expressed by
\begin{equation}\label{aii}
a_{k,k}(t)=\frac{d}{dt}\ln\frac{\tau_k(t)}{\tau_{k-1}(t)},
\end{equation}
where $\tau_k(t)=[\exp(tL^0)]_k$ is the $k$-th $\tau$-function as defined in \eqref{eq:tau1}.
\end{proposition} 
\begin{proof}
First recall 
from the proof of 
Proposition \ref{prop:LUsolution}
that $\dot{b}b^{-1}=(L)_{\ge 0}$. 
Then the diagonal elements $b_{k,k}$
of the matrix $b$ 
satisfy the equation,
\[
\frac{db_{k,k}}{dt}=a_{k,k}b_{k,k}\qquad\text{for}\quad k=1,\ldots,n,
\]
which gives $a_{k,k}=\frac{d}{dt}\ln b_{k,k}$.
  From the decomposition $\exp(tL^0)=ub$, we have
\[
\tau_k=[ub]_k=\prod_{i=1}^kb_{i,i}.
\]
This implies that $\displaystyle{b_{k,k}=\frac{\tau_k}{\tau_{k-1}}}$.
This completes the proof.
\end{proof}

If $M$ is a matrix, we define the notation with multi-time variables $\mathbf{t}=(t_1,\ldots,t_{n-1})$
\begin{equation} \label{eq:Theta}
\Theta_{M}(\mathbf{t}):=\sum_{j=1}^{n-1} M^j t_j.
\end{equation}

The proof of Proposition \ref{prop:LUsolution} 
can be easily extended to give the following 
solution to the f-KT hierarchy.
\begin{proposition}\label{prop:LUsolution2}
Consider the LU-factorization
\begin{equation}
\exp(\Theta_{L^0}(\mathbf{t}))=
u(\mathbf{t})b(\mathbf{t})\qquad \text{with}
\quad u(\mathbf{t})\in U^-,~ b(\mathbf{t})\in B^+. 
\end{equation}
The solution $L(\mathbf{t})$ of the f-KT hierarchy is then given by
\[
L(\mathbf{t})=u(\mathbf{t})^{-1} L^0u(\mathbf{t})=
b(\mathbf{t})L^0b(\mathbf{t})^{-1}.
\]
\end{proposition}

This motivates the following definition of the $\tau$-functions
for the 
f-KT hierarchy.
\begin{equation}\label{eq:tau2}
\tau_k(\mathbf{t})=[\exp(\Theta_{L^\0}(\mathbf{t}))]_k\qquad\text{with}\quad
\mathbf{t}=(t_1,\ldots,t_{n-1}),
\end{equation}
where $L^{\mathbf{0}}=L(\mathbf{0})$, the initial matrix of $L(\mathbf{t})$.
As before, the LU-factorization
of Proposition \ref{prop:LUsolution2}
exists if and only if each $\tau$-function
$\tau_k(\mathbf{t})$ is nonzero.

\begin{remark}\label{rem:solution}
An explicit formula for each entry $a_{i,j}(\t)$ 
of $L(\t)$ has been obtained in \cite{AvM, KY} in
terms of the $\tau$-functions and their derivatives with respect to $t_j$'s.  However, in this paper we need only the formula
for the diagonal elements 
given in \eqref{aii} with $t=t_1$.
We included a direct proof of this formula
in order to keep 
the paper self-contained.
\end{remark}

Note that 
if one identifies $t_1=x,~t_2=y$ and $t_3=t$, 
\eqref{eq:tau2} gives the $\tau$-function for the KP equation,
see e.g. \cite{BK}. Then $\tau_k$ is 
associated with a point of the Grassmannian $Gr_{k,n}$, 
and the set of $\tau$-functions
$(\tau_1,\ldots,\tau_{n-1})$ is associated with a point of the flag variety.
The solution space of the f-KT hierarchy is naturally given by the complete flag
variety.

In Section \ref{sec:solKT}, we will associate an 
initial matrix $L^{\mathbf{0}}$ to each point 
on the tnn flag variety.  We obtain in this way a large family of 
regular solutions of the f-KT hierarchy.


\section{The solution of the f-KT hierarchy from the tnn flag
variety}\label{sec:solKT}

In this section we discuss the behavior of solutions of 
the f-KT hierarchy when the initial point is associated to a point on the tnn flag variety. 
By using the 
decomposition \eqref{decomposition} of
the tnn flag variety $(G/B^+)_{\ge0}$, we determine the asymptotic form of 
the matrix $L(\t)$ when the first time variable $t=t_1$ goes to $\pm \infty$.  
Each asymptotic form is a particular fixed point of the f-KT lattice.
We then extend the asymptotic analysis to the case with the multi-times: 
one can reach any fixed point by sending the multi-time variables
to infinity in a particular direction.

We first illustrate how to embed the f-KT flow into the flag variety.

\subsection{The companion embedding of the f-KT flow into the flag variety}\label{sec:embedding}

Let us first recall the companion embedding of the iso-spectral variety $\mathcal{F}_{\Lambda}$ into the flag variety $G/B^+$ defined in \cite{FH}
\begin{align}\label{eq:companion}
c_{\Lambda}:&~~\mathcal{F}_{\Lambda}~\longrightarrow~G/B^+\\
&\hskip0.2cm L\hskip0.2cm\longmapsto ~~ u\,B^+ \nonumber
\end{align}
where $u\in U^-$ is the unique element given by  the decomposition $L=u^{-1}C_{\Lambda}u$ with
the companion matrix
\begin{equation}\label{C}
C_{\Lambda}=\begin{pmatrix}
0&1                &    0               &   \cdots  &     0      \\
0&0                &     1             &     \ddots  &  \vdots       \\
\vdots      & \vdots          &  \ddots   &     \ddots  & \vdots   \\
0&  0               &     \cdots              &  0  &     1         \\
\pm\sigma_n  &  \mp\sigma_{n-1} &  \cdots  &  \sigma_2&0 \\
\end{pmatrix} ~\in ~\epsilon+\mathfrak{b}^-.
\end{equation}
Here the $\sigma_i$'s are obtained from the characteristic polynomial $\text{det}(\lambda I-L)=\sum_{i=0}^n(-1)^i\sigma_i\lambda^{n-i}$ with $\sigma_0=1$,
that is, the $\sigma_i$'s are the elementary symmetric polynomials of 
the eigenvalues $\{\lambda_1,\ldots,\lambda_n\}$:
\[
\sigma_1=\sum_{j=1}^n\lambda_j=0,\quad \sigma_2=\sum_{i<j}\lambda_i\lambda_j,\quad \sigma_3=\sum_{i<j<k}\lambda_i\lambda_j\lambda_k,\qquad \cdots \qquad \sigma_n=\prod_{i=1}^n\lambda_i.
\]



By Proposition \ref{prop:LUsolution2}, each f-KT flow is represented
by the map 
$$Ad_{u(\t)^{-1}}~: ~L^0\, \longrightarrow\, L(\t).$$
We then have the following (see also \cite{FH, CK, KS08}).

\begin{proposition}\label{prop:companion}
Each f-KT flow maps to the flag variety as
\begin{equation}\label{companionE}
\begin{CD}
L^\0  @>c_{\Lambda}>> u_0\,B^+\\
@V Ad_{u(\t)^{-1}}VV @VVV \\
L(\t) @>c_{\Lambda}>>\quad\left\{
\begin{array}{lll}
     ~~ u_0\,u(\t)\,B^+ \\[0.5ex]
      = u_0 \exp(\Theta_{L^\0}(\t))\,B^+\\[0.5ex]
     =\exp(\Theta_{C_{\Lambda}}(\t))\,u_0\,B^+
\end{array}\right.
\end{CD}
\end{equation}
where $L^\0=u_0^{-1}C_{\Lambda}u_0$.  
That is, the initial matrix $L^\0$ determines the element $u_0\in U^-$, and each f-KT flow corresponds to an $\exp(\Theta_{C_{\Lambda}}(\t))$-orbit on 
the flag variety with the initial point $u_0B^+$.
\end{proposition}

\begin{remark} 
By Proposition \ref{prop:fixedpoint}, a fixed point of the f-KT lattice 
has the form
$L=(L)_{\ge0}$. This implies that if $L$ is a fixed point, then $\exp(\Theta_{L}(\mathbf{t}))\in B^+$, hence $u(\mathbf{t})$ is the identity matrix.
That is, a fixed point of the f-KT flow in a space $\mathcal{F}_{\Lambda}$ is the fixed point of the $\exp(\Theta_{C_{\Lambda}}(\mathbf{t}))$-action in the flag variety
$G/B^+$.

\end{remark}


\subsection{The full Kostant-Toda flow on $\mathcal{R}_{v,w}^{>0}$}\label{sec:guL}
In this section we associate to each matrix $g\in G^{>0}_{\v_+,\w}$
(representing a point of $\mathcal{R}_{v,w}^{>0}$) an initial matrix 
$L^\0$ for the f-KT hierarchy.  
We then express the $\tau$-functions of the f-KT hierarchy 
with the 
initial matrix $L^\0$ in terms of $g$.
Since the solution of the f-KT hierarchy can be given 
in terms of the $\tau$-functions in  \eqref{eq:tau2} (recall Remark \ref{rem:solution}),
this allows one to 
express the solution in terms of $g$.

The main result of this section is the following.

\begin{proposition}\label{prop:gsolution}
To each matrix $g\in G^{>0}_{\v_+,\w}$
we can associate an initial matrix 
$L^\0 \in \mathcal{F}_{\Lambda}$, 
defined by $L^\0 = u_0^{-1} C_{\Lambda} u_0$, where 
$C_{\Lambda}$ is given by \eqref{C}, and 
$u_0 \in U^-$ and $b_0 \in B^+$ are uniquely determined by 
the equation $Eg = u_0 b_0$. 
For this choice of $g$ and $L^\0$, the $\tau$-functions for 
the f-KT hierarchy
with initial matrix $L^\0$ are given by 
\begin{equation}\label{tauEg}
\tau_k(\mathbf{t})=[\exp(\Theta_{C_{\Lambda}}(\mathbf{t}))u_0]_k=d_k\left[E\exp\left(\Theta_{\Lambda}(\mathbf{t})\right)g\right]_k,
\end{equation}
where $d_k=[b_0^{-1}]_k$.
\end{proposition}

\begin{remark}
Note that if we use Proposition \ref{prop:gsolution} 
to associate $L^0$ to $g$, and we subsequently apply the companion embedding
to $L^0$, then we will obtain the point $u_0 B^+ = Egb_0^{-1} B^+ = Eg  B^+$
of the flag variety.  This is actually a point on the 
\emph{totally positive part} 
$\mathcal R_{e,w_0}^{>0}$ of the
flag variety.
\end{remark}


The following lemma implies that the construction of  $L^\0$ in 
Proposition \ref{prop:gsolution} is well-defined.

\begin{lemma}\label{Eg}
For each $g\in G^{>0}_{\v_+,\w}$,
the product $Eg$ has the LU-factorization, that is, there exist unique $u_0 \in U^-$ and $b_0 \in B^+$ such that
$E g =u_0\, b_0.$
\end{lemma}
\begin{proof}
We calculate the principal minors $[E\,g]_k$ of $Eg$
for $k=1,\dots,n$.
We have that \[
[E\,g]_k = \left|\begin{pmatrix}
1&1&\cdots &1\\
\vdots&\vdots&\ddots&\vdots\\
\lambda_1^{k-1}&\lambda_2^{k-1}&\cdots&\lambda_n^{k-1}
\end{pmatrix}
\,\begin{pmatrix}
g_{1,1}&\cdots & g_{1,k}\\
g_{2,1}&\cdots&g_{2,k}\\
\vdots&\ddots&\vdots\\ 
g_{n,1}&\cdots&g_{n,k}
\end{pmatrix}
\right|.
\]
Since $g\in G_{\v_+,\w}^{>0}$, Lemma \ref{lem:flagTNN} implies that 
all the $k\times k$ minors of the leftmost $k$ columns are non-negative
(and since $g \in G/B^+$, at least one of them is positive).
Also since $\lambda_1 < \cdots < \lambda_n$, all the $k\times k$ minors of the top $k$ rows of the Vandermonde
matrix are positive.  Therefore the Binet-Cauchy
lemma implies that $[E\,g]_k>0$ for all $k=1 \dots n$, and hence
$Eg$ has the LU-factorization.
\end{proof}
 
We now complete the proof of Proposition \ref{prop:gsolution}.
\begin{proof}
%
Our initial matrix is given by 
\[
L^\0=u_0^{-1}C_{\Lambda}u_0=u_0^{-1}E\Lambda E^{-1}u_0, 
\]
where we have used the diagonalization, $C_{\Lambda}=E\,\Lambda\,E^{-1}$.  
We then have
\begin{align*}
\exp\left(\Theta_{L^\0}(\mathbf{t})\right)&=u_0^{-1}\exp(\Theta_{C_{\Lambda}}(\t))\,u_0\\
&=u_0^{-1}E\,\exp\left(\Theta_{\Lambda}(\mathbf{t})\right)\,E^{-1}\,u_0\\
&=u_0^{-1}E\,\exp\left(\Theta_{\Lambda}(\mathbf{t})\right)\,g\,b_0^{-1}.
\end{align*}

Therefore  by \eqref{eq:tau2}, the $\tau$-functions of the 
f-KT hierarchy are given by 
\begin{align*}
\tau_k(\mathbf{t})&=
\left[\exp\left(\Theta_{L^\0}(\mathbf{t})\right)\right]_k
= \left[ u_0^{-1} E \exp\left(\Theta_{\Lambda}(\mathbf{t})\right)gb_0^{-1}\right]_k\\
&= \left[ E \exp\left(\Theta_{\Lambda}(\mathbf{t})\right)gb_0^{-1}\right]_k =d_k\left[E\exp\left(\Theta_{\Lambda}(\mathbf{t})\right)g\right]_k,
\end{align*}
where $d_k=[b_0^{-1}]_k$.  
\end{proof}

\begin{remark}\label{rem:torus}
From the proof above we see that 
$\exp(\Theta_{C_{\Lambda}}(\mathbf{t}))u_0=E\exp(\Theta_{\Lambda}(\mathbf{t}))gb_0^{-1}$.
This implies that the f-KT flow gives a (non-compact) torus action on
the flag variety.  More precisely, the torus $(\R_{>0})^n$ acts by 
$\exp(\Theta_{\Lambda}(\mathbf{t}))$ on the basis vectors consisting of
the columns of the Vandermonde matrix $E$, that is, we have
$\exp(\Theta_{L^0}(\mathbf{t}))u_0B^+=E\exp(\Theta_{\Lambda}(\mathbf{t}))gB^+$.
Note here that the torus $\exp(\Theta_{\Lambda}(\mathbf{t}))$ acts on the point $gB^+$.
\end{remark}

\begin{definition}\label{def:EI}
For $I=\{i_1,i_2,\ldots,i_k\}$, we set
$E_{i_j}(\mathbf{t}):=\exp\theta_{{i_j}}(\mathbf{t})$ with $\theta_{i_j}(\mathbf{t})=\sum_{m=1}^{n-1}\lambda_{i_j}^mt_m$, and define
\[
E_I(\mathbf{t}):=
\prod_{\ell<m}(\lambda_{i_m}-\lambda_{i_\ell})\,e^{\theta_I(\mathbf{t})}
\quad\text{with}\quad \theta_I(\mathbf{t})=\sum_{j=1}^k\theta_{i_j}(\mathbf{t}).
\]
\end{definition}

Note that $E_I(\mathbf{t})$ is always positive.

\begin{corollary}\label{cor:tauA}
Use the notation of Proposition \ref{prop:gsolution}.
Recall from \eqref{eq:pi} that $A_k = \pi_k(g)$
consists of the leftmost $k$ columns of $g$.
Then we can write the $\tau$-function as 
\begin{equation}\label{tau}
\tau_k(\mathbf{t})=d_k \sum_{I \in {[n] \choose k}} \Delta_I(A_k) E_I(\mathbf{t}).
\end{equation}
\end{corollary}

\begin{proof}
By Proposition \ref{prop:gsolution}, we have that 
\begin{align*}
\tau_k(\mathbf{t})&=d_k\left[E\exp\left(\Theta_{\Lambda}(\mathbf{t})\right)g\right]_k \\
&= d_k \left[
\begin{pmatrix}
1 &1& \cdots & 1\\
\lambda_1 & \lambda_2&\cdots & \lambda_n \\
\vdots &\vdots&\ddots  & \vdots \\
\lambda_1^{n-1} &\lambda_2^{n-1}& \cdots & \lambda_{n}^{n-1}
\end{pmatrix}
 \begin{pmatrix}
e^{\theta_{1}(\mathbf{t})} &&&\\
&e^{\theta_{2}(\mathbf{t})}&&\\
&&\ddots&\\
&&&e^{\theta_{n}(\mathbf{t})}
\end{pmatrix}
g~\right]_k\\
&= d_k \left|
\begin{pmatrix}
e^{\theta_{1}(\mathbf{t})} & e^{\theta_{2}(\t)}&\cdots & e^{\theta_{n}(\mathbf{t})}\\
\vdots &\vdots&\ddots  & \vdots \\
\lambda_1^{k-1} e^{\theta_{1}(\mathbf{t})} & \lambda_2^{k-1}e^{\theta_{2}(\t)}& \cdots & \lambda_1^{k-1} e^{\theta_{n}(\mathbf{t})}
\end{pmatrix}
 \begin{pmatrix}
g_{11} &\cdots & g_{1k}\\
g_{21} & \cdots & g_{2k}\\
\vdots  &\ddots  & \vdots \\
g_{n1}& \cdots & g_{nk}
\end{pmatrix}
\right|\\
&=d_k \sum_{I \in {[n] \choose k}} E_I(\mathbf{t}) \Delta_I(A_k),
\end{align*}
where in the last step, we have used the Binet-Cauchy lemma.
\end{proof}

When the initial 
matrix $L^0$ comes from a point in $G_{\v_+,\w}^{>0}$,
the f-KT flow is complete.

\begin{proposition}\label{prop:regular}
Let $g\in G_{\v_+,\w}^{>0}$, and $Eg=u_0b_0$
as in Proposition  \ref{prop:gsolution}.
Let $L(0)=L^0=u_0^{-1}C_{\Lambda}u_0$.  Then
the solution $L(\t)$ of the f-KT hierarchy is regular for all 
$\t=(t_1,\dots,t_{n-1})\in\R^{n-1}$.
\end{proposition}
\begin{proof}
Recall that by Proposition \ref{prop:LUsolution2},
the solution of the f-KT lattice
is given by $L(\t)=u^{-1}(\t)L^0u(\t)$,  
where $u(\t)\in U^-$ is obtained from the LU-factorization of the matrix
$\exp(\Theta_{L^0}(\mathbf{t}))=u(\t)b(\t)$ 
for $b(t)\in B^+$.
The $\tau$-functions are the principal minors of 
$\exp(\Theta_{L^0}(\mathbf{t}))$,
and by Corollary \ref{cor:tauA} and Lemma \ref{lem:flagTNN}
they are all positive.  Therefore
the LU-factorization exists and is 
unique for each $\t\in\R^{n-1}$.
This implies that 
$L(\t)$ is regular for all $\t\in\R^{n-1}$, that is, 
the f-KT flow is complete.
\end{proof}

\begin{remark}
If a $\tau$-function vanishes at a fixed multi-time $\t=\t_*$, then 
the LU-factorization of $\exp(\Theta_{L^0}(\mathbf{t}_*))$ fails, i.e.
$\exp(\Theta_{L^0}(\mathbf{t}_*))\in U^-\dot{z}B^+$ for some $z \in \Sym_n$ 
such that $z \neq e$.
This means that the f-KT flow becomes singular and
hits the  boundary of the top Schubert cell,
see \cite{CK, FH}. The set of times $\t_*$ where the $\tau$-function
vanishes is called the \emph{Painlev\'e divisor}.  
It is then quite interesting to identify  $z$
in terms of
the initial matrix $L^0$ given in the general form of \eqref{L}.
This will be discussed in a future work.
\end{remark}

\begin{remark}
The $\tau$-function in \eqref{tau} has the Wronskian structure, that is, if we define
the functions $\{f_1,\ldots,f_k\}$ by
\[
(f_1(\mathbf{t}),\ldots,f_k(\mathbf{t})):=(E_1(\mathbf{t}),\ldots,E_n(\mathbf{t}))\,A_k^T,
\]
then we have
\[
\tau_k(\mathbf{t})=d_k\,\text{Wr}(f_1(\mathbf{t}),\ldots,f_k(\mathbf{t})),
\]
where the Wronskian is for the $t_1$-variable.  
Furthermore, if we identify the first three variables
as $t_1=x,~t_2=y$ and $t_3=t$ in \eqref{tau},
then we obtain the $\tau$-function for the KP equation
which gives rise to soliton solutions of the KP equation from the Grassmannian
$Gr_{k,n}$
\cite{KW3}.
\end{remark}

\subsection{Asymptotic behavior of the f-KT lattice}

In this section we consider the asymptotics of the solution 
$L(t)$ to the f-KT lattice as $t=t_1$ goes to $\pm \infty$,
and where $L(0) = L^0$ is the initial matrix from Proposition \ref{prop:gsolution}.
To analyze the asymptotics, we 
use the $\tau$-functions  $\tau_k(\mathbf{t})$ from
\eqref{tau} with all other times $t_j$ being constants for $j\ge 2$.
In Section \ref{subsec:fKThierarchy}, we extend the asymptotic results of this section to the case of the f-KT hierarchy
in the $\mathbf{t}$-space.

Recall that we have a fixed order $\lambda_1 < \dots < \lambda_n$
on the eigenvalues, and that 
$z \cdot [k]$ denotes the ordered set $\{z(1), z(2),\dots, z(k)\}$.  

\begin{lemma}\label{asympt-tau}
Let $g\in G^{>0}_{\v_+,\w}$.  Then 
each $\tau_k({t})$-function from \eqref{tau} 
has the following asymptotic behavior:
\[
\tau_k({t})~\longrightarrow~\left\{\begin{array}{lll}
E_{w \cdot [k]}({t})\quad&\text{as}\quad t\to\infty\\[1.0ex]
E_{v \cdot [k]}({t})\quad&\text{as}\quad t\to-\infty
\end{array}\right.
\]
\end{lemma}
Note that in Lemma \ref{asympt-tau} we are ignoring the coefficients
of the exponentials $E_{w \cdot [k]}({t})$ and 
$E_{v \cdot [k]}({t})$. 

\begin{proof}
Recall that 
\[
\mathcal{M}(A_k)=\left\{I\in\binom{[n]}{k}~\Big|~\Delta_I(A_k)\ne 0\right\}.
\]
Since $A_k = \pi_k(g)$ and $g\in U^-\dot{v}\cap B^+\dot{w}B^+$ (by Theorem 
\ref{t:parameterization}),  the lexicographically maximal and minimal elements in $\mathcal{M}(A_k)$
are respectively given by $w \cdot [k]$ and $v \cdot [k]$.

Recall from Definition \ref{def:EI} 
that $E_j({t}) = \exp(\theta_{j}({t}))$. 
Since $\lambda_1 < \dots < \lambda_n$,
we have 
\begin{align*}
&E_1\ll E_2\ll \cdots \ll E_n,\quad\text{as}\quad t\to \infty, \\[0.5ex]
&E_1\gg E_2\gg \cdots\gg E_n,\quad\text{as}\quad t \to -\infty,
\end{align*}
which implies the lemma.
\end{proof}

Our main theorem on the asymptotic behavior of the 
matrix $L(\t)$ is the 
following.
\begin{theorem}\label{asympt}
Let $g\in G^{>0}_{\v_+,\w}$ and let $L^0\in \mathcal{F}_{\Lambda}$ 
be the initial matrix as given by Proposition \ref{prop:gsolution}.  Then
the diagonal elements $a_{k,k}=a_{k,k}(t)$ of $L(t)$ satisfy
\[
a_{k,k}~\longrightarrow~\left\{\begin{array}{lll}
\lambda_{w(k)}\quad&\text{as}\quad t\to\infty\\[1.0ex]
\lambda_{v(k)}\quad&\text{as}\quad t\to-\infty
\end{array}\right.
\]
Furthermore,  
$L(t)$ approaches a fixed point of the f-KT flow as $t\to\pm\infty$:
we have
$(L)_{<0}\to 0$ and 
\[
L(t)~\longrightarrow~\left\{\begin{array}{lll}
\epsilon+\text{diag}(\lambda_{w(1)},\lambda_{w(2)},\dots,\lambda_{w(n)})\quad&\text{as}\quad t\to\infty\\[1.5ex]
\epsilon+\text{diag}(\lambda_{v(1)},\lambda_{v(2)},\dots,\lambda_{v(n)})\quad &\text{as}\quad t\to-\infty
\end{array}\right.
\]
\end{theorem}
\begin{proof}  We consider only the case for $t\to\infty$;
the proof is the same in the other case.

First recall from \eqref{aii} that
\[
a_{k,k}=\frac{d}{dt}\ln\frac{\tau_k}{\tau_{k-1}}.
\]
Then from Lemma \ref{asympt-tau}, we have
\[
a_{k,k}~\longrightarrow~\frac{d}{dt}\ln\frac{E_{w\cdot[k]}}{E_{w\cdot[k-1]}}=\frac{d}{dt}\ln E_{w(k)}=\lambda_{w(k)}\quad\text{as}~t\to\infty.
\]
Note  that $u_{k,k}:=a_{k,k}-\lambda_{w(k)}$ decays \emph{exponentially}
to zero as $t\to\infty$, i.e. we have $\displaystyle{\frac{da_{k,k}}{dt}\to 0}$ exponentially.  This decay property of the diagonal elements 
is key for proving the convergence $(L)_{<0}\to0$.

Now we show $(L)_{<0}\to 0$. From \eqref{eq:f-KTcomp} with $\ell=0$, 
we have
\[
\frac{da_{k,k}}{dt}=a_{k+1,k}-a_{k,k-1},\qquad\text{for}\quad k=1,\ldots,n,
\]
where by convention we have $a_{i,j}=0$ if $i=n+1$ or $j=0$.
Since the left hand side decays exponentially, the case 
$k=1$ gives  
$a_{2,1}\to0$.  Now by induction on $k$ 
we find that 
all the elements of the first sub-diagonal decay exponentially, that is,
we have
\[
a_{k+1,k}~\longrightarrow~0\quad \text{for  }~k=1,\ldots,n-1.
\]

From \eqref{eq:f-KTcomp} with $\ell=1$, we have
\[
\frac{da_{k+1,k}}{dt}=a_{k+2,k}-a_{k+1,k-1}+(a_{k+1,k+1}-a_{k,k})a_{k+1,k},\qquad\text{for}\quad k=1,\ldots,n-1,
\]
and hence  
the elements of the second sub-diagonal of $(L)_{<0}$ are given by
\begin{align*}
a_{k+2,k}&=a_{k+1,k-1}+\frac{da_{k+1,k}}{dt}-(a_{k+1,k+1}-a_{k,k})a_{k+1,k}.
\end{align*}
Now using induction on $k$, 
all the functions on the right-hand side decay exponentially, and hence
the elements $a_{k+2,k}$ for $k=1,\ldots,n-2$ also decay exponentially.
(The base case with $k=1$
shows the decay of $a_{3,1}$.)

Continuing this argument for $\ell>1$ yields
$(L)_{<0}\to0$.  This completes the proof.
\end{proof}
\begin{example}
Consider the $\mathfrak{sl}_5$ f-KT lattice.
Let $g\in G^{>0}_{\v_+,\w}$ where $\v_+$ and $\w$ are given by
\[
\w=s_2s_3s_1s_4s_3s_2\qquad \text{and}\qquad \v_+=s_2\,1\,1\,s_4s_3\,1.
\]
Since $w\cdot(1,2,3,4,5)=(3,5,1,4,2)$ and $v\cdot(1,2,3,4,5)=(1,3,5,2,4)$, we have the following asymptotic matrices,
\[
L(t)~\longrightarrow~\left\{\begin{array}{lll}
\epsilon+\text{diag}(\lambda_1,\lambda_3,\lambda_5,\lambda_2,\lambda_4)\quad&\text{as}\quad t\to-\infty\\[1.0ex]
\epsilon+\text{diag}(\lambda_3,\lambda_5,\lambda_1,\lambda_4,\lambda_2)\quad &\text{as}\quad t\to\infty
\end{array}\right.
\]
\end{example}

\begin{definition}
We say that a flow $L(t)$ \emph{has the asymptotic form $(v,w)$}
for $(v,w) \in \Sym_n\times \Sym_n$ 
if when $t\to-\infty$ (resp. $t\to\infty$) the 
Lax matrix $L(t)$ 
tends to $\epsilon+ \diag(\lambda_{v(1)},\ldots,\lambda_{v(n)})$ (resp.
$\epsilon+\diag(\lambda_{w(1)},\ldots,\lambda_{w(n)})$).
\end{definition}

Then Theorem \ref{asympt} implies the following.
\begin{corollary}\label{cor:flow}
We consider f-KT flows coming from points of $(G/B^+)_{\geq 0}$.
If $v \leq w$ are comparable elements in 
$\Sym_n$, then there exists a regular f-KT flow whose asymptotic form is 
$(v,w)$.  Moreover, the dimension of the space of such flows 
is $\ell(w)-\ell(v)$.
\end{corollary}
\begin{proof}
By Proposition \ref{prop:regular},
a flow coming from $g \in G_{\v_+, \w}^{>0}$
is regular.  Theorem \ref{t:parameterization}
implies that if we consider all flows $L(t)$ with initial matrix
$L^0$
coming from some $g$ representing a point of $(G/B^+)_{\geq 0}$, 
the dimension of the space of flows with asymptotic form 
$(v,w)$ is precisely $\ell(w)-\ell(v)$.
\end{proof}



\begin{remark}
The fact that the space of flows
with the asymptotic form $(v,w)$ has dimension $\ell(w)-\ell(v)$ is 
similar to \cite[Corollary 3.3]{CSS} for the case
of the full symmetric Toda lattice.  We will discuss the relation between
the f-KT lattice and the full symmetric Toda lattice in Section \ref{sec:symmetricToda}.
\end{remark}
\subsection{Asymptotic behavior of the f-KT flow in the $\mathbf{t}$-space}\label{subsec:fKThierarchy}
We now consider the solution $L(\mathbf{t})$ to the f-KT hierarchy with the multi-time variables $\mathbf{t}=(t_1,\ldots,t_{n-1})$.
The main result in this section is the following.
\begin{proposition}\label{prop:fixedpointz}
 Let $g\in G_{\v_+,\w}^{>0}$, and let $z\in \Sym_n$.  
Assume that $v \leq z \leq w$, or equivalently 
(by Lemma \ref{minors}) $\Delta_{z\cdot[k]}^k(g)>0$ for all $k=1,\ldots,n-1$.
Then there exists a direction $\mathbf{t}(s)=s \mathbf{c}$ 
with a constant vector $\mathbf{c}=(c_1,\ldots,c_{n-1})$ such that 
\begin{equation}\label{eq:limit}
L(\mathbf{t}(s))~\longrightarrow~\epsilon+\text{diag}(\lambda_{z(1)},\ldots,\lambda_{z(n)})\quad
\text{as}\quad s\to\infty.
\end{equation}
Moreover, for any $\mathbf{c}$ such that 
$E_{z(1)}(\mathbf{c}) > E_{z(2)}(\mathbf{c}) > \dots >
E_{z(n)}(\mathbf{c}),$ we have \eqref{eq:limit}.
In other words, when we take such a limit, $L(\mathbf{t})$ approaches 
the fixed point
$\epsilon+\text{diag}(\lambda_{z(1)},\ldots,\lambda_{z(n)}).$

\end{proposition}

To prove the proposition, we use the following lemma.
\begin{lemma}\label{lem:direction}
For any permutation $z \in \Sym_n$, 
one can find a multi-time $\mathbf{c}=(c_1,\dots, 
c_{n-1})\in \R^{n-1}$
such that 
$E_{z(1)}(\mathbf{c}) > E_{z(2)}(\mathbf{c}) > \dots >
E_{z(n)}(\mathbf{c}).$
\end{lemma}
\begin{proof}
Recall that $E_i(\mathbf{t}) = \exp\theta_{i}(\mathbf{t})$ with
$\theta_{i}(\mathbf{t})=\sum_{m=1}^{n-1}\lambda_i^mt_m$.
Define the function $\ell_i: \R\times \R^{n-1} \to \R$
by 
\[
\ell_i(t_0,\mathbf{t})= t_0+
\lambda_i t_1 + \lambda_i^2 t_2 + \dots + \lambda_i^{n-1} t_{n-1} = t_0+\theta_{i}(\mathbf{t}) =(t_0,t_1,\ldots,t_{n-1})\cdot E_i^0,
\]
where $E_i^0$ is the $i$-th column vector of the Vandermonde matrix $E$.
To prove the lemma, it suffices to find a point 
$(t_0,\mathbf{t})$ such 
that $\ell_{z(1)}(t_0,\mathbf{t}) > \ell_{z(2)}(t_0,\mathbf{t}) > 
\dots > \ell_{z(n)}(t_0,\mathbf{t})$, which also implies
that $\theta_{{z(1)}}(\mathbf{t})>\theta_{{z(2)}}(\mathbf{t})>\cdots>\theta_{{z(n)}}(\mathbf{t})$, and hence
$E_{z(1)}(\mathbf{t}) > E_{z(2)}(\mathbf{t}) > \dots > E_{z(n)}(\mathbf{t}).$

Let $(r_1,\dots,r_n) \in \R^n$ be any point such that 
$r_{z(1)} > \dots > r_{z(n)}$.  Then since $E$ is an invertible
matrix, we can solve the equation
$(t_0, t_1,\dots, t_{n-1}) \cdot E = (r_1,\dots,r_n)$.  This proves 
the lemma.
\end{proof}

\begin{remark}
In the proof of Lemma \ref{lem:direction}, note that $E_i^0$ is 
the normal vector to the plane given by $\ell_i(t_0,\mathbf{t})= \constant$.  
Since the $\lambda_i$'s are distinct, the $E_i^0$'s are linearly 
independent.
Therefore 
the hyperplane arrangement 
\begin{equation}\label{eq:arrangement}
\ell_i(t_0,\mathbf{t}) - \ell_j(t_0,\mathbf{t})
=0 \quad \text{for}\quad 1\le i<j\le n
\end{equation}
is a \emph{braid arrangement} which has $n!$ regions in its complement, 
indexed by the permutations
$z \in \Sym_n$ 
(the permutation specifies the total order of the values that the 
functions $\ell_i$ take on that region).
Since the arrangement \eqref{eq:arrangement}
does not depend on the $t_0$ variable, 
the $(n-1)$-dimensional $\mathbf{t}$-space is divided into
$n!$ convex polyhedral cones, each  labeled  by some $z\in \Sym_n$
specifying the total order
of the
$\theta_{{i}}(\mathbf{t})$'s.  See Figure \ref{fig:braid} and Example \ref{ex:braid} below.
\end{remark}

We now prove Proposition \ref{prop:fixedpointz}.
\begin{proof}
Choose a point $\mathbf{c}$ in the $\mathbf{t}$-space such that 
$E_{z(1)}(\mathbf{c})>\cdots>E_{z(n)}(\mathbf{c})$; by 
Lemma \ref{lem:direction}, one exists.  Recall that  $E_I(\mathbf{t})=\prod_{\ell<m}(\lambda_{i_m}-\lambda_{i_{\ell}})\prod_{j=1}^kE_{i_j}(\mathbf{t})$
for $I=\{i_1,\ldots,i_k\}$.
This implies that $E_{z\cdot[k]}(\mathbf{c})$ dominates the other exponentials
in the $\tau_k$-function \eqref{tau} at the point $\mathbf{c}$. 
Then we consider the direction $\mathbf{t}(s)=s\mathbf{c}$, and take the limit
$s\to\infty$.  This gives
\[
\tau_k(\mathbf{t}(s))~\longrightarrow~d_k\Delta_{z\cdot[k]}^k(g)E_{z\cdot[k]}(\mathbf{t}(s))\quad\text{as}\quad s\to\infty.
\]
Now using the formula for $a_{k,k}(\mathbf{t})$ in \eqref{aii} with $t=t_1$, one can 
follow the proof of Theorem \ref{asympt} to show that 
$a_{k,k}(\t(s)) \to \lambda_{z(k)}$ as $s\to\infty$.
We also have that 
$\displaystyle{\frac{\partial a_{k,k}}{\partial t_1}\Big|_{\t=\t(s)}\to 0}$ as $s\to\infty$,
which allows us to use the proof of Theorem \ref{asympt} 
to show that 
$(L)_{<0}\to0$, i.e. $L(\t(s))$ approaches the fixed point as $s\to\infty$.
\end{proof}

\begin{example}\label{ex:braid}
Consider the case of the $\mathfrak{sl}_3(\R)$ f-KT lattice,
shown in 
Figure \ref{fig:braid}.  
The $\mathbf{t}$-space 
is divided into 
$|\mathfrak{S}_3|=6$ convex polyhedral cones by the braid arrangement $\ell_i-\ell_j=0$ for
$1\le i<j\le 3$.  
The cone defined by 
$\ell_i>\ell_j>\ell_k$ is labeled by
the permutation $z\in\mathfrak{S}_3$ such that 
$z(1)=i,$ $z(2)=j$, and $z(3)=k$.
In this example, we have the following limits of $t_2$ with $t_1$=constant:
\[
L(t_1,t_2)~\longrightarrow~\left\{\begin{array}{lll}
\epsilon+\text{diag}(\lambda_3,\lambda_1,\lambda_2)\quad&\text{as}\quad t_2\to\infty\\[1.0ex]
\epsilon+\text{diag}(\lambda_2,\lambda_1,\lambda_3)\quad &\text{as}\quad t_2\to-\infty
\end{array}\right.
\]
Thus each cone corresponds to a fixed point of the f-KT hierarchy, and 
the polygon connecting those fixed points is a hexagon (see the right 
of Figure \ref{fig:braid}).
This polygon is the so-called \emph{permutohedron} of $\Sym_3$. 
We will discuss the permutohedron and a generalization of it 
in Section \ref{sec:moment}.
\begin{figure}[h]
\centering
\includegraphics[height=4.8cm]{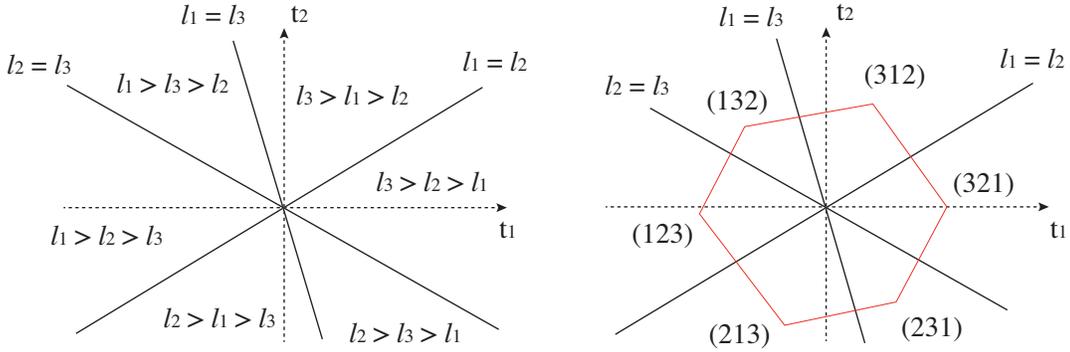}
\caption{The braid arrangement for the $\mathfrak{sl}_3(\R)$ f-KT lattice.
Here $\ell_i-\ell_j=(\lambda_i-\lambda_j)t_1+(\lambda_i^2-\lambda_j^2)t_2$
with $\lambda_1<\lambda_2<\lambda_3$.
There are six cones, and each cone has the ordering $\ell_{z(1)}>\ell_{z(2)}>\ell_{z(3)}$ with some $z\in\mathfrak{S}_3$.  Each vertex $(i,j,k)$ 
of the polytope corresponds to the cone with
$\ell_i>\ell_j>\ell_k$.
\label{fig:braid}}
\end{figure}

\end{example}

\section{The moment polytope of the f-KT lattice}\label{sec:moment}

In this section we define the moment map for the flag variety,
and study the image of the moment map on the f-KT flows coming from
the tnn flag variety.  In Theorem \ref{thm:polytope}
we show that we obtain in this way certain
convex polytopes which generalize the permutohedron. 

\subsection{The moment map for Grassmannians and flag varieties}

Let $\mathsf{L}_i$ denote a weight of 
the standard representation of $\mathfrak{sl}_n$, i.e. 
for $h=\text{diag}(h_1,\ldots,h_n)\in\mathfrak{h}$, 
$\mathsf{L}_i(h)=h_i$.
Then recall that the weight space $\mathfrak{h}_{\R}^*$
is the dual of the Cartan subalgebra $\mathfrak{h}_{\R}$, i.e. 
\[
\mathfrak{h}_{\R}^*:=\text{Span}_{\R}\left\{\mathsf{L}_1,\ldots, \mathsf{L}_n~\Big|~\sum_{j=1}^n\mathsf{L}_j=0\right\}\cong\R^{n-1}.
\]

For $I = \{i_1,\dots,i_k\}$, we set 
$\mathsf{L}(I)
=\mathsf{L}_{i_1}+\mathsf{L}_{i_2}+\cdots+\mathsf{L}_{i_k} \in \mathfrak{h}_{\R}^*.$
The \emph{moment map} for the Grassmannian
$\mu_k:Gr_{k,n}\to\mathfrak{h}_{\R}^*$ is defined by 
\begin{equation}\label{momentGrassmannian}
\mu_k(A_k):=\frac{\sum_{I\in\mathcal{M}(A_k)}\left|\Delta_{I}(A_k)\right|^2\mathsf{L}(I)}{\sum_{I\in\mathcal{M}(A_k)}\left|\Delta_I(A_k)\right|^2},
\end{equation}
see e.g. \cite{GS, GGMS, Shipman}.  

We recall the following fundamental result of
Gelfand-Goresky-MacPherson-Serganova \cite{GGMS} on the moment map
for the Grassmannian (which in turn uses the convexity theorem of 
Atiyah \cite{A:82} and Guillemin-Sternberg \cite{GuS}).

\begin{theorem} \cite[Section 2]{GGMS} \label{thm:convex}
If $A_k \in Gr_{k,n}$ and we consider the action of the 
torus $(\C^*)^n$ on $Gr_{k,n}$ (which rescales columns of the 
matrix representing $A_k$), then the closure of the image of the moment map
applied to the torus orbit of $A_k$ is a
convex polytope 
\begin{equation}\label{Gamma}
\Gamma_{\M(A_k)} = \conv\{\mathsf{L}(I) \ \mid \ 
\Delta_I(A_k) \neq 0 \text{ i.e. }I\in \M(A_k)\}
\end{equation}
called a 
\emph{matroid polytope}, whose vertices
correspond to the
fixed points of the action of the torus.  
\end{theorem}

\begin{remark}
In representation theory, this polytope is a weight polytope of the fundamental representation of $\mathfrak{sl}_n$ on $\wedge^kV$, where $V$ is
the standard representation.  
\end{remark}

\begin{corollary}\label{cor:convexity}
If $A_k \in Gr_{k,n}$ and we consider the action of the 
\emph{positive} torus $(\R_{>0})^n$ on $Gr_{k,n}$, the conclusion of 
Theorem \ref{thm:convex} still holds. 
\end{corollary}

\begin{proof}
If $T_u:=\text{diag} (u_1,\dots,u_n)$ with $u_i \in \C^*$ then
for $I = \{i_1,\dots,i_k\}$, we have that  
$\Delta_I(A_kT_u) = 
u_{i_1} u_{i_2} \dots u_{i_k} \Delta_I(A_k)$.
It's clear from the definition that 
$\mu_k(A_kT_u)$ depends only on the magnitudes of each of the 
$u_i$'s.  Therefore the image of the moment map applied to the torus orbit
of $A_k$ is the same if we restrict to the positive torus.
\end{proof}

One may extend the moment map from the Grassmannian to the flag variety
$G/B^+$ \cite{GS}.  First 
recall that 
the flag variety $G/B^+$ has a projective embedding in 
$\mathbb{P}(V)\times \mathbb{P}(\wedge^2V)\times \cdots \times\mathbb{P}(\wedge^{n-1}V)$ which is the image of $G$ acting on the highest weight
vector 
$[e_1\otimes(e_1\wedge e_2)\otimes\cdots\otimes(e_1\wedge\cdots\wedge e_{n-1})]$ (which has weight 
$n\mathsf{L}_1+(n-1)\mathsf{L}_2+\cdots+\mathsf{L}_{n}=\sum_{k=1}^{n-1}\mathsf{L}(I_k)$ with $I_k:=\{1,\ldots,k\}$).

The \emph{moment map} for the flag variety 
$\mu: G/B^+ \to \mathfrak{h}_{\R}^*$ is defined by
\[
\mu(g):= \sum_{k=1}^{n-1}\mu_k(A_k),\qquad \text{ where }\quad A_k = \pi_k(g).
\]

If we let $gB^+$ vary over all points in $G/B^+$, 
we obtain the \emph{moment polytope}, which is also called the 
\emph{permutohedron} $\Perm_n$.
For $(i_1,i_2,\dots, i_n) \in \Sym_n$, define the weight 
\[
\mathsf{L}_{i_1,i_2,\ldots,i_n}:=n\mathsf{L}_{i_1}+(n-1)\mathsf{L}_{i_2}+\cdots+\mathsf{L}_{i_n}=\sum_{k=1}^{n}\mathsf{L}(I_k),
\]
where $I_k:=\{i_1,\ldots,i_k\}$.  (Note $\mathsf{L}(I_n)=\mathsf{L}_1+\cdots+\mathsf{L}_n=0$.)
We identify  $\mathsf{L}_{1,2,\dots,n}$ with the highest weight vector.
By definition, the permutohedron 
 $\Perm_n$ is 
the convex hull of the weights
$\mathsf{L}_{\pi}:=\mathsf{L}_{i_1,\ldots,i_n}$ as $\pi(1,\ldots,n)=(i_1,\dots,i_n)$ varies over $\Sym_n$, i.e.
\[
\Perm_n=\CH\{\mathsf{L}_{\pi}\in\mathfrak{h}_{\R}^*~|~\pi\in \Sym_n\}.
\]
\begin{example}
Consider the case of $G=\text{SL}_3(\R)$. The permutohedron $\Perm_3$
is given by Figure \ref{fig:permutoS3}.
\begin{figure}[h]
\centering
\includegraphics[height=5cm]{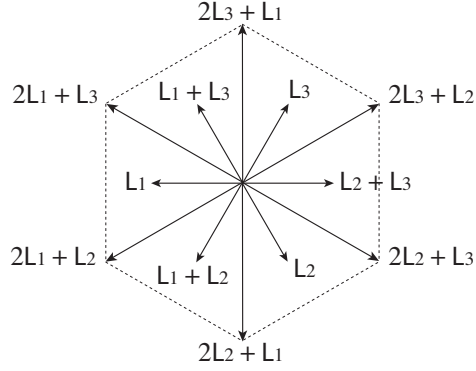}
\caption{The permutohedron $\Perm_3$, which is a hexagon.  Each vertex is given by the
weight $\mathsf{L}_z:=\mathsf{L}_{z(1),z(2),z(3)}=2\mathsf{L}_{z(1)}+\mathsf{L}_{z(2)}=3\mathsf{L}_{z(1)}+2\mathsf{L}_{z(2)}+\mathsf{L}_{z(3)}$ for some $z\in\Sym_3$.
\label{fig:permutoS3}}
\end{figure}

\end{example}

\subsection{The moment polytope of the f-KT flows}
Recall from Section \ref{sec:guL} that to each point 
of the tnn flag variety (represented by some point 
$g \in G_{\v_+,\w}^{>0}$), one can associate a unique element $u_0 \in U^-$,
which in turn determines an initial point $L^0$ of a flow.  Throughout
this section we fix $g$, $u_0$, and $L^0$ accordingly.  Furthermore 
we let $A_k=\pi_k(g)$ be the span of the leftmost $k$ columns of $g$, as in 
\eqref{eq:pi}.
Our main goal here is to compute
the image of the
moment map 
$\mu:G/B^+\to\mathfrak{h}^*_{\R}$ when applied to the f-KT flow
$\exp(\Theta_{C_\Lambda}(\mathbf{t}))$
on the point $u_0B^+$ of the flag variety 
described in \eqref{companionE}.

The following lemma will help us compute the image of the moment map.

\begin{lemma}\label{lem:momentprojection}
Let $\mathcal{E}:=\{E_i^0:i=1,\ldots,n\}$ be the ordered set of column vectors of $E$, i.e.
$E_i^0=(1,\lambda_i,\ldots,\lambda_i^{n-1})^T$.  Then the projection
$\tilde\pi_k$ of \eqref{basis} on $\mathbb{P}(\wedge^k\R^n)$ with the basis $\mathcal{E}$ is given by
\[
\tilde\pi_k(\exp(\Theta_{C_\Lambda}(\mathbf{t}))u_0)  =
A_ke^{\Theta_{\Lambda}(\mathbf{t})}.
\]
\end{lemma}

\begin{proof}
We compute  $\exp(\Theta_{C_\Lambda}(\mathbf{t}))u_0\cdot e_1\wedge\cdots\wedge e_k$ on $\wedge^k\R^n$ with the basis $\mathcal{E}$.  The first
equality below comes from Remark \ref{rem:torus}, and recall that
$d_k=[b_0^{-1}]_k$.
\begin{align*}
&\exp(\Theta_{C_\Lambda}(\mathbf{t}))u_0\cdot e_1\wedge\cdots\wedge e_k  =
E\,e^{\Theta_{\Lambda}(\mathbf{t})}\,gb_0^{-1} \cdot e_1\wedge\cdots\wedge e_k\\
=&\sum_{1\le i_1<\cdots<i_k\le n}Ee^{\Theta_{\Lambda}(\mathbf{t})}e_{i_1}\wedge
\cdots\wedge e_{i_k}\langle e_{i_1}\wedge\cdots\wedge e_{i_k}, g b_0^{-1}\cdot
e_1\wedge\cdots\wedge e_k\rangle\\
=&d_k \sum_{1\le i_1<\cdots<i_k\le n}Ee^{\Theta_{\Lambda}(\mathbf{t})}e_{i_1}\wedge
\cdots\wedge e_{i_k}\langle e_{i_1}\wedge\cdots\wedge e_{i_k}, g \cdot
e_1\wedge\cdots\wedge e_k\rangle\\
=&d_k\sum_{1\le i_1<\cdots<i_k\le n}\Delta_{i_1,\ldots,i_k}(A_k) Ee^{\Theta_{\Lambda}(\mathbf{t})}e_{i_1}\wedge\cdots\wedge e_{i_k}\\
=&d_k\sum_{1\le i_1<\cdots<i_k\le n}\Delta_{i_1,\ldots,i_k}(A_k) e^{\theta_{i_1,\ldots,i_k}(\mathbf{t})} E\cdot e_{i_1}\wedge\cdots\wedge e_{i_k}\\
=&d_k\sum_{1\le i_1<\cdots<i_k\le n}\Delta_{i_1,\ldots,i_k}(A_ke^{\Theta_{\Lambda}(\mathbf{t})})  E_{i_1}^0\wedge\cdots\wedge E_{i_k}^0,
\end{align*} 
where $\theta_{i_1,\ldots,i_k}(\mathbf{t})=\sum_{j=1}^k\theta_{i_j}(\mathbf{t})$ with
$\theta_j(\mathbf{t})=\sum_{m=1}^{n-1}\lambda_j^mt_m$.
This completes the proof.
\end{proof}

Lemma \ref{lem:momentprojection} implies that 
the Pl\"ucker coordinates for the f-KT flow are very simple
when computed in the 
basis $\mathcal{E}=\{E_1^0,\dots,E_n^0\}$.

We then define the moment map $\mu$ in this basis as follows.
First write
$\varphi(\mathbf{t};g):=
\mu(\exp(\Theta_{C_\Lambda}(\mathbf{t}))u_0)$ and  
$\varphi_k(\mathbf{t};g):= 
\mu_k(\tilde\pi_k(\exp(\Theta_{C_\Lambda}(\mathbf{t}))u_0))=\mu_k(A_ke^{\Theta_{\Lambda}(\mathbf{t})})$ with $\pi_k(g)=A_k$. 
Then we have
\begin{align}\label{momentT}
\varphi(\mathbf{t};g)&=\sum_{k=1}^{n-1}\varphi_k(\mathbf{t};g)\qquad\text{with}\quad
\varphi_k(\mathbf{t};g)=\sum_{I\in\mathcal{M}(A_k)}\alpha_{I}^k(\mathbf{t};g)\,\mathsf{L}(I),\\ \nonumber
&\text{and} \qquad \alpha_I^k(\mathbf{t}; g)
=\frac{\left(\Delta_{I}(A_ke^{\Theta_{\Lambda}(\mathbf{t})})\right)^2}
{\sum_{J\in\mathcal{M}(A_k)}\left(\Delta_J(A_k e^{\Theta_{\Lambda}(\mathbf{t})})\right)^2}.
\end{align}
Note here that  $0<\alpha_{I}^k(\mathbf{t};g)<1$ and $\sum_{I\in\mathcal{M}(A_k)}\alpha_I^k(\mathbf{t};g)=1$ for each $k$.

\begin{definition}
We define the \emph{moment map image of the f-KT flow} for  $g\in G_{\v,\w}^{>0}$ to be the 
set $$\QQ_g=\overline{\left\{\varphi(\mathbf{t};g)~|~\mathbf{t}\in\R^{n-1}\right\}}:=\overline{\bigcup_{\mathbf{t}\in\R^{n-1}}\varphi(\mathbf{t};g)}.$$
Here the closure is taken 
using the usual topology of the Euclidian norm on $\mathfrak{h}_\R^*\cong\R^{n-1}$. 
\end{definition}

\begin{proposition}\label{Qkg}
For each $k$,
define $\QQ^k_g:=\overline{\{\varphi_k(\mathbf{t};g)~|~\mathbf{t}\in\R^{n-1}\}}.$
Then $\QQ^k_g$ equals the corresponding matroid polytope 
from \eqref{Gamma}, i.e. 
\[
\QQ^k_g=\Gamma_{\M(A_k)} \quad \text{ where }~A_k = \pi_k(g).
\]
\end{proposition}
\begin{proof}
 Recall that
$\varphi_k(\mathbf{t}; g) = 
\mu_k(A_ke^{\Theta_{\Lambda}(\mathbf{t})})$.
Since the $\lambda_i$'s are distinct, 
as $\mathbf{t}$ varies over $(\R)^{n-1}$, the diagonal matrices
$T_{\Lambda}:=e^{\Theta_{\Lambda}(\mathbf{t})}$ cover all points 
in the positive torus.   The result now follows from 
Corollary \ref{cor:convexity}.
\end{proof}

\begin{corollary}\label{cor:Minkowski}
Let $g\in G_{\v_+,\w}^{>0}$.
Then the  moment map image $\QQ_g$ of 
the f-KT flow for $g$  is a
Minkowski sum of matroid polytopes.  
More specifically, for $A_k=\pi_k(g)$, $k=1,\ldots,n-1$, we have
$$\QQ_g = \sum_{k=1}^{n-1} 
\Gamma_{\M(A_k)}.
$$
\end{corollary}
\begin{proof}
This follows from \eqref{momentT} and Proposition \ref{Qkg}.
\end{proof}

We also define a certain polytope which sits inside the permutohedron.
\begin{definition}
Let $v$ and $w$ be two permutations in $\Sym_n$ such that $v \leq w$.
We define the \emph{Bruhat interval polytope} associated to $(v,w)$ to be the following convex hull:
\begin{equation*}
\mathsf{P}_{v,w}:=\CH\{\mathsf{L}_z \in\mathfrak{h}_{\R}^* \ \vert~ v \leq z \leq w\}.
\end{equation*}
In other words, this is the convex hull of all permutation vectors
corresponding to permutations $z$ lying in the Bruhat interval $[v,w]$.
In particular, 
if $w=w_0$ and $v=e$, then we have $\mathsf{P}_{e,w_0} = \Perm_n$.

\end{definition}

The main result of this section is the following.
\begin{theorem}\label{thm:polytope}
Let $g\in G_{\v_+,\w}^{>0}$.
Then the  moment map image of 
the f-KT flow for $g$   is the convex polytope $\mathsf{P}_{v,w}$, i.e.
\[
\QQ_g = \mathsf{P}_{v,w}.
\]
\end{theorem}

\begin{proof}
Recall from Corollary \ref{cor:Minkowski} that 
$\QQ_g$ is a 
Minkowski sum of the matroid polytopes
$\Gamma_{\M(A_k)}$ for $1 \leq k \leq n-1$.
Therefore $\QQ_g$ is a convex polytope.

Now consider an arbitrary vertex $V$ of $\QQ_g$.
We claim that $V$
has the form $\mathsf{L}_z$ for some $z \in \Sym_n$.
Since $\QQ_g$ is the Minkowski sum of the polytopes
$\Gamma_{\M(A_k)}$, $V$ can be written as a sum 
of vertices of those polytopes, i.e.  for $I_k\in \M(A_k)$, $k=1,\ldots,n-1$, 
the vertex $V$ has the form 
\begin{equation}\label{Minkowskisum}
\mathsf{L}(I_1) + \mathsf{L}(I_2) + \dots + \mathsf{L}(I_{n-1}).
\end{equation}
To prove the claim, it suffices to show 
that $I_1 \subset I_2 \subset \dots \subset
I_{n-1}.$

By \eqref{momentT}, any point of 
${\left\{\varphi(\mathbf{t};g)~|~\mathbf{t}\in\R^{n-1}\right\}}$
has the form 
\begin{equation}\label{eq:form}
\sum_{k=1}^{n-1} \sum_{I \in \M(A_k)} \alpha_I^k \hspace{2pt}
 \mathsf{L}(I)
\end{equation}
where  $0 < \alpha_I^k < 1$ and $\sum_{I\in\M(A_k)}\alpha_I^k=1$ for each $k$.
Because each $A_k$ is the projection $\pi_k(g)$ 
of the same element $g$,
it follows that whenever the coefficient of $\mathsf{L}(I)$ is 
nonzero (i.e. $I\in\M(A_k)$),  for each 
$j<k$ 
there exists some $j$-element subset $J \subset I$ such that the coefficient of 
$\mathsf{L}(J)$ is nonzero (i.e. $J \in \M(A_j)$).
We call the latter property the \emph{flag} property of points
in ${\left\{\varphi(\mathbf{t};g)~|~\mathbf{t}\in\R^{n-1}\right\}}$.

But now the vertex $V$ is a limit of points of the form
\eqref{eq:form} which have the flag property, and also 
$V$ has the form \eqref{Minkowskisum}.  In particular, for each 
$k$, precisely one coefficient $\alpha_I^k$ equals $1$ and the rest are zero.
It follows that $I_1 \subset I_2 \subset \dots \subset I_{n-1}$.
This proves the claim.


%

Now suppose that $z$ is a permutation such that 
$v \leq z \leq w$.  
We will show that $\mathsf{L}_z \in \QQ_g$.
By Lemma \ref{minors}, we have $\Delta_{z\cdot [k]}^{k}(g)\ne0$ for each $k$,
and hence $z\cdot[k] \in \M(A_k)$.
It follows that 
$\sum_{i=1}^k \mathsf{L}_{z(i)}=\mathsf{L}(z\cdot[k])$
is a vertex of $\Gamma_{\M(A_k)}$,
and hence that 
$\sum_{k=1}^n(n-k)\,\mathsf{L}_{z(k)}
= \mathsf{L}_z$
is a point of $\QQ_g$.
This shows that 
if $v \leq z \leq w$ then 
$\mathsf{L}_z \in \QQ_g$.

Conversely, suppose that for some permutation $z$ 
we have $\mathsf{L}_z \in \QQ_g$.  
Then since $\mathsf{L}_z= 
\sum_{k=1}^n (n-k) \mathsf{L}_{z(k)}$,
$$\varphi(\mathbf{t};g)=
\sum_{k=1}^{n-1} 
\sum_{I\in\mathcal{M}(A_k)}\alpha_{I}^k(\mathbf{t};g)\,\mathsf{L}(I),$$
 and $0<\alpha_{I}^k(\mathbf{t};g)<1$,
it follows that 
$z\cdot[k] \in \M(A_k)$ for $1 \leq k \leq n-1$.
Therefore $\Delta^k_{z \cdot [k]}(g) \neq 0$ 
for $1 \leq k \leq n-1$, and hence by Lemma 
\ref{minors}, we have that 
$v \leq z \leq w$.
We have shown that 
$\mathsf{L}_{z} \in \QQ_g$ if and only if $v \leq z \leq w$.
Therefore $\QQ_g$ is precisely the convex hull of the 
points $\mathsf{L}_z$ where $v \leq z \leq w$.
%
%
\end{proof}

\begin{corollary}\label{cor:BIPM}
The Bruhat interval polytope $\mathsf{P}_{v,w}$ is a Minkowski
sum of matroid polytopes 
$$\mathsf{P}_{v,w} = \sum_{k=1}^{n-1} \Gamma_{\M_k}.$$
Here $\M_k$ is the matroid defining the cell 
of $(Gr_{k,n})_{\geq 0}$ that we obtain by projecting the 
cell $\mathcal{R}_{v,w}^{>0}$ of $(G/B^+)_{\geq 0}$ to 
$(Gr_{k,n})_{\geq 0}$. 
\end{corollary}

\begin{proof}
This follows from Corollary \ref{cor:Minkowski} and Theorem 
\ref{thm:polytope}.
\end{proof}

By Proposition \ref{prop:fixedpointz}, each weight vector $\mathsf{L}_{i_1,\ldots,i_n}$ can be associated to the
ordered set of eigenvalues,
\[
\mathsf{L}_{i_1,\ldots,i_n}\quad\Longleftrightarrow\quad (\lambda_{i_1},\lambda_{i_2},\ldots,\lambda_{i_n}).
\]
For example, the highest weight for the permutohedron is given by
\[
\mathsf{L}_{1,2,\ldots,n}=\sum_{k=1}^n(n-k)\mathsf{L}_k\quad\Longleftrightarrow\quad(\lambda_1,\lambda_2,\ldots,\lambda_n).
\]
which corresponds to the asymptotic form of $\text{diag}(L)$ with $v=e$ for $t\to-\infty$.

\begin{example}
Consider the $\mathfrak{sl}_4(\R)$ f-KT hierarchy. We take 
\[
w=s_2s_3s_2s_1 \text{ and } v=s_3,
\]
which gives
\[
w\cdot(1,2,3,4)=(4,1,3,2) \text{ and } v\cdot(1,2,3,4)=(1,2,4,3).
\]
There are eight permutations $z$ satisfying
$v\le z\le w$, i.e.
\[
v=s_3,\quad s_3s_2,\quad s_2s_3,\quad s_3s_1,\quad s_3s_2s_1,\quad s_2s_3s_1,\quad s_2s_3s_2,\quad w=s_2s_3s_2s_1.
\]
We illustrate the moment polytopes in Figure \ref{fig:Mpoly}.  The vertices are labeled by the index set ``$i_1i_2i_3i_4$'' of the eigenvalues $(\lambda_{i_1},\lambda_{i_2},\lambda_{i_3},\lambda_{i_4})$. The vertex with the white circle indicates the asymptotic form of $\text{diag}(L)$ for
$t\to-\infty$ (i.e. $(1,2,4,3)=v\cdot(1,2,3,4)$ in the right figure), and the black one indicates the asymptotic form for $t\to\infty$ (i.e. $(4,1,3,2)=w\cdot(1,2,3,4)$ in the right figure).
\begin{figure}[h]
\centering
\includegraphics[height=5.3cm]{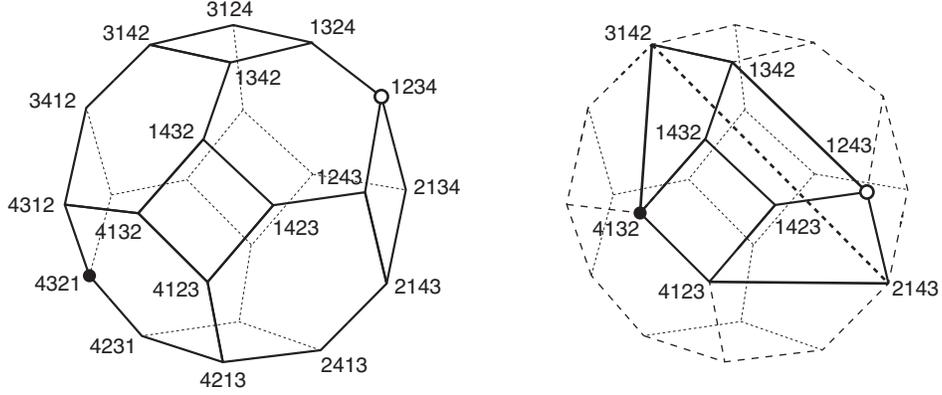}
\caption{Some moment polytopes from the
 $\mathfrak{sl}_4(\R)$ f-KT hierachy.   The left figure is the
permutohedron $\mathsf{P}_{e,w_0}=\Perm_4$, and the right one is the 
Bruhat interval polytope $\mathsf{P}_{v,w}$ with $w=s_2s_3s_2 s_1$ and $v=s_3$.
\label{fig:Mpoly}}
\end{figure}
\end{example}

\section{The full symmetric Toda hierarchy}
\label{sec:symmetricToda}

In this section we discuss the full symmetric Toda hierarchy for a symmetric Lax matrix $\mathcal{L}$ in connection with the  f-KT hierarchy for $L\in\epsilon+\mathfrak{b}^-$.
We first review work of Bloch-Gekhtman \cite{BG} which provides 
a map from the f-KT hierarchy to the symmetric Toda hierarchy.
Although this map is not always defined, in Theorem \ref{thm:symm} 
we will give a sufficient condition, 
in terms of points $g \in G_{\v_+,\w}^{>0}$ of the tnn flag variety, 
under which the map is defined.  This allows us to map 
the corresponding 
flows in the f-KT hierarchy to flows in the symmetric Toda hierarchy,
and to compute the moment map images for such flows in the symmetric
Toda hierarchy.

In the process we will explain the commutative diagram,
\[
\begin{CD}
L @ > Ad_{\beta^{-1}}> > \mathcal{L}\\
@ V Ad_{u} VV @ VV Ad_{q} V \\
C_{\Lambda}@ > Ad_{E^{-1}}> > \Lambda
\end{CD}
\]
where $E$ is the Vandermonde matrix \eqref{E}, and 
$u\in U^-,~ \beta\in B^+$ and $q\in \text{O}_n(\mathbb{R})$ 
are obtained uniquely from $g\in G_{\v_+,\w}^{>0}$.

\subsection{The full symmetric Toda hierarchy and the f-KT hierarchy}
\label{sec:symmfKT}

We now explain the construction of Bloch-Gekhtman \cite{BG} for mapping flows
$L = L(\mathbf{t})$ 
for the f-KT hierarchy to flows $\mathcal{L} = \mathcal{L}(\mathbf{t})$
for the symmetric Toda hierarchy.  Note that 
this construction is not defined at those $\t$ where $L(\mathbf{t})$ is 
singular.

First recall from \eqref{eq:companion} 
that for each $L\in \mathcal{F}_{\Lambda}$,
there is
a unique element $u\in U^-$ such that
\[
L=u^{-1}C_{\Lambda}u,
\]
where $C_{\Lambda}$ is the companion matrix. 
Since $C_{\Lambda} = E \Lambda E^{-1}$, we have
\[
L=\text{Ad}_\gamma\Lambda=\gamma\Lambda \gamma^{-1},\qquad \text{where}\quad \gamma:=u^{-1}E.
\]

\begin{lemma}\label{gamma}
Let $\dot{\gamma}$ denote the derivative of $\gamma=\gamma(\mathbf{t})$ with respect to $t_m$. 
Then we have
\[
\dot{\gamma} \gamma^{-1}=(L^m)_{\ge 0}-L^m=-(L^m)_{<0}.
\]
\end{lemma}
\begin{proof}
Note that 
 \begin{align*}
 \dot{L}
 &=\dot{\gamma}\gamma^{-1}L-L\dot{\gamma}\gamma^{-1}=[\dot{\gamma}\gamma^{-1},\,L],
 \end{align*}
 where we have used the fact that 
 $\frac{\partial}{\partial t_m}{\gamma}^{-1}=-\gamma^{-1} \dot{\gamma}\gamma^{-1}$.
   From $\gamma=u^{-1}E$, we also have
 \[
 \dot{\gamma}\gamma^{-1}=-u^{-1}\dot{u} \,\in\, \mathfrak{u}^-.
 \] 

Let $F = (L^m)_{\geq 0} - \dot{\gamma} \gamma^{-1}$ and $K= L^m-F$.
To prove the lemma, it suffices to show that $F = L^m$, or equivalently,
$K=0$.
The fact that $\dot{L} = [\dot{\gamma}\gamma^{-1},\,L]$
and $\dot{L}=[(L^m)_{\ge0},L]$ (the Lax equation) implies that 
$[F,L]=0$.  Then Lemma \ref{lem:commutation} implies that 
$F=\sum_{i=0}^{n-1}a_iL^i$ for some 
constants $a_i$.

Now we show $K=L^m-F=0$.
Using the definition of $F$ and $\dot{\gamma} \gamma^{-1} \in \mathfrak{u}^-$, 
we have $(F)_{\geq 0} = (L^m)_{\geq 0}$, and hence $(K)_{\geq 0} = 0$.
Therefore all the eigenvalues of $K$ are $0$.  Also note that using the Cayley-Hamilton lemma,  one can write
\[
K=\sum_{i=0}^{n-1}b_iL^i
\]
for unique $b_i$'s.  Then we have $0=b\,E$ with $b=(b_0,b_1,\ldots,b_{n-1})$.
This implies that $b=0$, which shows $K=0$.  Therefore we have $F=L^m$,
as desired.
\end{proof}

Following Bloch-Gekhtman \cite{BG}, we associate a symmetric matrix 
$\mathcal{L}$ to $L$:
\begin{equation}\label{symmL}
\mathcal{L}=\psi(L):=\text{Ad}_{\beta^{-1}}L=\beta^{-1}L\beta,\qquad \text{with}\quad \beta\in B^+.
\end{equation}
The matrix $\beta$ is given by Cholesky matrix factorization, i.e. 
for a fixed  $\gamma=u^{-1}E$,
\begin{equation}\label{beta}
\gamma\gamma^T=\beta\beta^T.
\end{equation}
This implies that 
\[
\mathcal{L}^T=\beta^T\gamma^{-T}\Lambda \gamma^T\beta^{-T}=\beta^{-1}\gamma\Lambda \gamma^{-1}\beta=\mathcal{L},
\]
where $x^{-T}=(x^{-1})^T=(x^T)^{-1}$.
\begin{remark}
Since $\gamma\gamma^T=u^{-1}EE^Tu^{-T}$ whose principal minors are all positive, the matrix $\gamma\gamma^T$ is
positive definite. This guarantees that the Cholesky factorization is unique.
\end{remark}

The following result shows that 
the map \eqref{symmL} transforms the f-KT hierarchy into the symmetric Toda hierarchy.
\begin{theorem}\cite[Theorem 3.1]{BG} \label{sToda}
If $L = L(\mathbf{t})$ satisfies the f-KT hierarchy, the symmetric matrix 
$\mathcal{L} = \mathcal{L}(\mathbf{t})$
in \eqref{symmL} satisfies the symmetric Toda hierarchy, i.e.
\[
\frac{\partial \mathcal{L}}{\partial t_m}=[\pi_{\mathfrak{so}}(\mathcal{L}^m),\mathcal{L}],\qquad\text{with}\quad
\pi_{\mathfrak{so}}(\mathcal{L}^m):=(\mathcal{L}^m)_{>0}-(\mathcal{L}^m)_{<0}.
\]
\end{theorem}
\begin{proof}
Let $\dot{\mathcal{L}}$ denote the $t_m$-derivative of $\mathcal{L}$.  
Using \eqref{symmL}
it is straightforward to check that 
\begin{align*}
\dot{\mathcal{L}}&=[B,\mathcal{L}],\qquad\text{where}\qquad B=-\beta^{-1}\dot\beta+\beta^{-1}\dot{\gamma}\gamma^{-1}\beta.
\end{align*}
 Using $\beta\beta^T=\gamma\gamma^T$ and
its $t_m$-derivative, one can show that $B$ is a skew-symmetric matrix, i.e. $B^T=-B$,
denoted by $B\in \mathfrak{so}(n)$.
Also using $L^m=(L^m)_{<0}+(L^m)_{\ge 0}$  and  Lemma \ref{gamma}, we have the decomposition
\begin{align*}
B&=-\beta^{-1}\dot{\beta}+\beta^{-1}\dot{\gamma}\gamma^{-1}\beta=-\beta^{-1}\dot{\beta}+\beta^{-1}\left[(L^m)_{\ge0}-L^m\right]\beta\\
&=-\beta^{-1}\dot{\beta}+\beta^{-1}(L^m)_{\ge 0}\beta-\mathcal{L}^m\\
&=\left(-\beta^{-1}\dot{\beta}+\beta^{-1}(L^m)_{\ge 0}\beta-(\mathcal{L}^m)_{\ge 0}\right)-(\mathcal{L}^m)_{<0}.
\end{align*}
Note that the first term $(-\beta^{-1}\dot{\beta}+\dots)$ in the last line belongs to $\mathfrak{b}^+$.  Then we use the fact that $B\in \mathfrak{so}(n)$ to conclude that $B=\pi_{\mathfrak{so}}(\mathcal{L}^m)$.
\end{proof}
\begin{remark}
It is known that the full symmetric Toda flow in Theorem \ref{sToda} is complete
(i.e. regular for all $\mathbf{t}$), see e.g. \cite{DLNT}.   However, 
the full-Kostant Toda flow with general initial data is not complete -- it can 
have a singularity.  Therefore, as noted in \cite{BG}, 
whenever the solution $L(\mathbf{t})$ becomes
singular,
the map from $L(\mathbf{t})$
to $\mathcal{L}(\mathbf{t})$ does not exist. 
In Theorem \ref{thm:symm} 
below we show that if the initial matrix $L^0$ comes from a point of 
the tnn flag variety, as in Section \ref{sec:guL}, 
then the map from $L(\mathbf{t})$ to $\mathcal{L}(\mathbf{t})$ 
exists for all 
$\mathbf{t}\in\R^{n-1}$.
\end{remark}


\subsection{The initial matrix $\mathcal{L}^0$ from $G^{>0}_{\v_+,\w}$}

Recall from Section \ref{sec:guL} that from each point 
$g \in G^{>0}_{\v_+,\w}$ we can construct an initial matrix $L^0$ for the 
full Kostant-Toda lattice.
By setting $\mathcal{L}^0=\beta_0^{-1}L^0\beta_0$, 
where $\beta_0$ represents $\beta$ at $t=0$,
we can also use 
$g\in G^{>0}_{\v_+,\w}$ 
to construct the initial data $\mathcal{L}^0=\mathcal{L}(0)$ for the full symmetric 
Toda hierarchy.

\begin{proposition}\label{lem:initialsymm}
Fix $g \in G^{>0}_{\v_+,\w}$ and define $L^0$ as in Section \ref{sec:guL}, i.e.
$L^0 = u_0^{-1} C_{\Lambda} u_0 = u_0^{-1} E \Lambda E^{-1} u_0$, where $u_0$ is determined by the equation
$Eg = u_0 b_0$.  Let $\mathcal{L}^0 = \psi(L^0)$ be
the symmetric matrix associated to $L^0$ by \eqref{symmL}.  Then 
\begin{equation*}
\mathcal{L}^0 = q_0^T \Lambda q_0,
\end{equation*}
where $q_0 \in \text{SO}_n(\R)$ is determined by the 
QR-factorization 
$g = q_0 r_0$, 
where $r_0 \in B^+$ is determined up to rescaling its columns by $\pm 1$.
\end{proposition}

\begin{proof}
We have that 
$\mathcal{L}^0 = \psi(L^0) = \beta_0^{-1} L^0 \beta_0 = \beta_0^{-1} u_0^{-1} E \Lambda E^{-1} u_0 \beta_0$,
where by \eqref{beta} with $\gamma_0=u_0^{-1}E$, we have 
$$\gamma_0\gamma_0^T=u_0^{-1}E(u_0^{-1}E)^T = \beta_0 \beta_0^T.$$
But then since $Eg = u_0 b_0$, we have
$\mathcal{L}^0 = \beta_0^{-1} b_0 g^{-1} \Lambda g b_0^{-1} \beta_0$ and 
$(b_0 g^{-1}) (b_0 g^{-1})^T = \beta_0 \beta_0^T,$ which implies that 
\begin{equation}
\label{orthogonal}
(\beta_0^{-1} b_0 g^{-1})(\beta_0^{-1} b_0 g^{-1})^T = I.
\end{equation}
This shows that $q_0 := (\beta_0^{-1} b_0 g^{-1})^{-1} \in \text{O}_n(\R)$,
and hence 
we have 
$$\mathcal{L}^0 = q_0^T \Lambda q_0.$$
If we set $r_0:=\beta_0^{-1} b_0 \in B^+$ then we have 
$$g = q_0 r_0 \qquad \text{ for }\quad q_0 \in \text{O}_n(\R), \quad r_0 \in B^+,$$ which 
is the  
QR-factorization of $g$.
(Note that the QR-factorization is unique up to 
the signs of the diagonal entries of $r_0$.)
Since we are working in $G = \SL_n(\R)$, 
$q_0$ in fact lies in $\text{SO}_n(\R)$.
\end{proof}

\subsection{The QR-factorization and the $\tau$-functions of the full symmetric Toda lattice}
Recall that we have the following isomorphism for the real flag variety:
\[
\text{SL}_n(\R)/B^+\cong \text{SO}_n(\R)/T_{\mathfrak{so}},\qquad\text{where}\quad
T_{\mathfrak{so}}:=\diag(\pm 1, \dots, \pm 1)
\]
is the maximal torus in $\text{SO}_n(\R)$.
When we associate an initial point $L^0$ for the full Kostant-Toda lattice
to a point $gB^+$ of the flag variety we are using the first point of view
on the flag variety.
Proposition \ref{lem:initialsymm} shows that our association of 
an initial point $\mathcal{L}^0$ for the symmetric Toda lattice is quite
natural from the second 
point of view.

Proposition \ref{prop:QR} below shows that one can use
the QR-factorization to solve the
full symmetric Toda lattice (see also \cite{Symes}).
\begin{proposition}\label{prop:QR}
The solution $\mathcal{L}(t)$ of the full symmetric Toda lattice is given by
\[
\mathcal{L}(t)=q(t)^{-1}\mathcal{L}^0 q(t)=r(t)\mathcal{L}^0 r(t)^{-1},
\]
where $q(t)\in \text{SO}_n(\R)$ and $r(t)\in B^+$ obtained by
the QR-factorization of the matrix $\exp(t\mathcal{L}^0)$, i.e.
\begin{equation}\label{QR-exp}
\exp(t\mathcal{L}^0)=q(t)r(t).
\end{equation}
\end{proposition}
\begin{proof}
Differentiating \eqref{QR-exp}, we have
\[
\mathcal{L}^0 \exp(t\mathcal{L}^0) = \mathcal{L}^0qr=qr\mathcal{L}^0=\dot{q}r+q\dot{r},
\]
where $\dot{q}$ means the derivative of $q(t)$.
Then we have
\[
q^{-1}\mathcal{L}^0q=r\mathcal{L}^0r^{-1}=q^{-1}\dot{q}+\dot{r}r^{-1}.
\]
Following the arguments in Proposition \ref{prop:LUsolution}, we can also show
\begin{equation*}\label{symmQR}
\mathcal{L}=q^{-1}\mathcal{L}^0q=r\mathcal{L}^0r^{-1}.
\end{equation*}
Here we have used the following projections of $\mathcal{L}$,
\begin{align}\nonumber
q^{-1}\dot{q}&=-\pi_{\mathfrak{so}}(\mathcal{L})=-(\mathcal{L})_{>0}+(\mathcal{L})_{<0}\\
\dot{r}r^{-1}&=\mathcal{L}+\pi_{\mathfrak{so}}(\mathcal{L})=\text{diag}(\mathcal{L})+2(\mathcal{L})_{>0}.\label{eq:r}
\end{align}
Those equations are needed when we compute the derivative of $\mathcal{L}$
as in Proposition \ref{prop:LUsolution}.
\end{proof}

As in the case of the f-KT hierarchy, one can also easily obtain the solution
of the full symmetric Toda hierarchy by extending the QR-factorization
\eqref{QR-exp} with multi-times $\mathbf{t}=(t_1,\ldots,t_{n-1})$.
\begin{proposition}\label{prop:QRsolution2}
Consider the QR-factorization
\begin{equation}\label{QR-exp2}
\exp(\Theta_{\mathcal{L}^0}(\mathbf{t}))=q(\mathbf{t})r(\mathbf{t})\qquad
\text{ with }
\quad q(\mathbf{t})\in \text{SO}_n(\R)\quad\text{and}\quad r(\mathbf{t})\in B^+.
\end{equation}
The solution $\mathcal{L}(\mathbf{t})$ of the full symmetric
Toda hierarchy is then given by
\[
\mathcal{L}(\mathbf{t})=q(\mathbf{t})^{-1}\mathcal{L}^0 q(\mathbf{t})=
r(\mathbf{t})\mathcal{L}^0r(\mathbf{t})^{-1}.
\]
\end{proposition}

We also have an analogue of Proposition \ref{prop:companion}
for the full symmetric Toda hierarchy.
\begin{proposition}
Consider 
the full symmetric Toda hierarchy, 
with the initial data
\[
\mathcal{L}^0=q_0^T\Lambda q_0,\qquad\text{with}\quad q_0\in \text{SO}_n(\mathbb{R}),
\]
where $q_0$ is obtained by the QR-factorization of $g\in G^{>0}_{\v_+,\w}$,
i.e. $g=q_0r_0$ with $r_0\in B^+$.
Using the QR-factorization \eqref{QR-exp2},
we have
$\mathcal{L}(\mathbf{t}) = q(\mathbf{t})^{-1} \mathcal{L}^0 
q(\mathbf{t}).$ Then the symmetric Toda flow maps to the flag variety
by the \emph{diagonal embedding} $d_{\Lambda}$ as in the following diagram.
\begin{equation}\label{commutative}
\begin{CD}
\mathcal{L}^\0  @> d_{\Lambda}>> q_0\,B^+\\
@V Ad_{q(\t)^{-1}}VV @VVV \\
\mathcal{L}(\t) @> d_{\Lambda}>>\quad\left\{
\begin{array}{lll}
     ~~ q_0\,q(\t)\,B^+ \\[0.5ex]
      = q_0 \exp(\Theta_{\mathcal{L}^\0}(\t))\,B^+\\[0.5ex]
     =\exp(\Theta_{{\Lambda}}(\t))\,q_0\,B^+
\end{array}\right.
\end{CD}
\end{equation}
Here the diagonal embedding from symmetric matrices $\mathcal{L}$ to $G/B^+$ takes
$q^T \Lambda q$ to $qB^+$, and the full symmetric Toda flow
corresponds to the $\exp(\Theta_{\Lambda}(\mathbf{t}))$-torus action on the point $q_0B^+$.
\end{proposition}

The main results of this section are Theorem \ref{thm:symm} on the asymptotic 
behavior of the full symmetric Toda lattice, and Theorem
\ref{thm:symmpolytope} on the moment polytope for the full symmetric Toda
hierarchy.

\begin{theorem}\label{thm:symm}
Let $g\in G_{\v^+,\w}^{>0}$.  Define 
$u_0$ by 
$Eg = u_0 b_0$ with $u_0\in U^-$ and $b_0 \in B^+$,
and define $q_0$ by
$g=q_0r_0$ with $q_0\in \text{SO}_n(\R)$ and
$r_0\in B^+$.
Let 
$L^0$ be the Hessenberg matrix 
$L^0=u_0^{-1}C_{\Lambda}u_0$, and 
let $\mathcal{L}^0$ be the symmetric matrix  $\mathcal{L}^0=
q_0^T\Lambda q_0$.
Then $L(t)$ is regular for all $t=t_1$, and the solution $\mathcal{L}(t)$ of 
the full symmetric Toda equation 
with the initial matrix $\mathcal{L}^0$
has the same asymptotic behavior as the solution $L(t)$ of the 
full Kostant-Toda equation  with
the initial matrix 
$L^0.$
More specifically, if
\[
{L}(t)\quad\longrightarrow\quad\epsilon+\text{diag}(\lambda_{z^{\pm}(1)},\ldots,\lambda_{z^{\pm}(n)})\qquad \text{as}~t\to \pm \infty
\]
for some $z^{\pm}\in\Sym_n$, 
then the corresponding solution $\mathcal{L}(t)$ satisfies
\[
\mathcal{L}(t)\quad\longrightarrow\quad\text{diag}(\lambda_{z^{\pm}(1)},\ldots,\lambda_{z^{\pm}(n)})\qquad \text{as}~t\to\pm \infty.
\]
Therefore we have that 
\[
\mathcal{L}(t)~\longrightarrow~\left\{\begin{array}{lll}
\text{diag}(\lambda_{v(1)},\lambda_{v(2)},\dots,\lambda_{v(n)})\quad&\text{as}\quad t\to-\infty\\[1.5ex]
\text{diag}(\lambda_{w(1)},\lambda_{w(2)},\dots,\lambda_{w(n)})\quad &\text{as}\quad t\to\infty
\end{array}\right.
\]

\end{theorem}
\begin{proof}
Proposition \ref{prop:regular}
implies that $L(t)$ is regular for all $t$.  
Now we use Theorem \ref{sToda}
and Proposition \ref{lem:initialsymm},
which imply that 
$\mathcal{L}(t) = \beta(t)^{-1} L(t) \beta(t)$ 
for $\beta(t) \in B^+$.  Note that the matrix 
$L^{\infty}:=\epsilon+\text{diag}(\lambda_{z(1)},\ldots,\lambda_{z(n)})$ is in $\mathfrak{b}^+$
and hence conjugating it by an element of $B^+$ results in an element of $\mathfrak{b}^+$.
The only symmetric matrices which lie in $\mathfrak{b}^+$ are diagonal matrices, so 
the corresponding limit point of $\mathcal{L}(t)$ is diagonal.
Finally, the only diagonal matrix which is conjugate by $B^+$ to 
$L^{\infty}$ is 
$\text{diag}(\lambda_{z(1)},\ldots,\lambda_{z(n)}).$
The second part of the theorem now follows from Theorem 
\ref{asympt}.
\end{proof}

\begin{remark}
Note that the equation $\dot{r}=(\diag(\mathcal{L})+2(\mathcal{L})_{>0})r$ from \eqref{eq:r}
gives the following equation for the diagonal elements of the matrix $r$:
\[
\frac{dr_{k,k}}{dt}=\alpha_{k,k}r_{k,k},
\]
where $\alpha_{k,k}$ is the $k$-th element of the diagonal part of $\mathcal{L}$.
That is, we have
\[
\alpha_{k,k}=\frac{d}{dt}\ln r_{k,k}.
\]
Also, from the QR-factorization $\exp(t\mathcal{L}^0)=qr$, one can write
\[
\exp(2t\mathcal{L}^0)=r^Tq^Tqr=r^Tr,
\]
where we have used $\exp(t\mathcal{L}^0)^T=\exp(t\mathcal{L}^0)=r^Tq^T=qr$
with $ q^Tq=I$.
Now we define the $\tau$-functions for the full symmetric Toda lattice as
the principal minors of $\exp(2t\mathcal{L}^0)$, i.e.
\[
\tau_k^{sym}(t)=[\exp(2t\mathcal{L}^0)]_k=[r^Tr]_k=\prod_{i=1}^{k}(r_{i,i})^2.
\]
This then gives the formula for the diagonal elements $\alpha_{k,k}$ of the symmetric Lax matrix $\mathcal{L}$, 
\[
\alpha_{k,k}=\frac{1}{2}\frac{d}{dt}(\ln r_{k,k}^2)=\frac{1}{2}\frac{d}{dt}\ln\frac{\tau_k^{sym}}{\tau_{k-1}^{sym}}.
\]
This formula can be used to give an alternative proof of Theorem \ref{thm:symm}.
\end{remark}

\begin{remark}
The first part of Theorem \ref{thm:symm}
recovers recent results of Chernyakov-Sharygin-Sorin, 
see \cite[Theorem 3.1 and Corollary 3.3]{CSS}.
\end{remark}

The following lemma will be helpful for proving Theorem 
\ref{thm:symmpolytope} below.
\begin{lemma}\label{lemma:AQ}
Let $g$ and $q_0$ be as in Proposition \ref{lem:initialsymm}.
Let $A_k$ and $Q_k$ be the $n\times k$ submatrices of $g$ and $q_0$ consisting of
the first $k$ column vectors of $g$ and $q_0$, respectively.
Then the matroids $\mathcal{M}(A_k)$ and $\mathcal{M}(Q_k)$ are the same.
\end{lemma} 

\begin{proof}
This can be shown by calculating the minors $\Delta_I$ with $I=\{i_1,\ldots,i_k\}$:
\begin{align*}
\Delta_I(A_k)&=\langle e_{i_1}\wedge\cdots\wedge e_{i_k},~g\cdot e_1\wedge\cdots\wedge e_k\rangle
=\langle e_{i_1}\wedge\cdots\wedge e_{i_k},~q_0\,r_0\cdot e_1\wedge\cdots\wedge e_k\rangle\\
&=(r_0^{11}\cdots r_0^{kk})\langle e_{i_1}\wedge\cdots\wedge e_{i_k},~q_0\cdot e_1\wedge\cdots\wedge e_k\rangle
=(r_0^{11}\cdots r_0^{kk})\Delta_I(Q_k),
\end{align*}
where the $r_0^{ii}$'s are the diagonal elements of the matrix $r_0$, 
which are all nonzero.
\end{proof}

\begin{theorem} \label{thm:symmpolytope}
Use the same hypotheses as in Theorem \ref{thm:symm}.
Then the closure of the 
image of the moment map for the full symmetric Toda hierarchy has 
the same Bruhat interval polytope $\mathsf{P}_{v,w}$ as does the full 
Kostant-Toda hierarchy.
\end{theorem}
\begin{proof} First note from \eqref{commutative} that the full symmetric Toda hierarchy gives the torus action $e^{\Theta_{\Lambda}(\mathbf{t})}$ on the flag variety. 
Then following the arguments on the moment map for the f-KT hierarchy,
one can see that the moment map for the full symmetric Toda hierarchy 
is given by
\begin{align*}\label{momentT-symm}
\tilde\varphi(\mathbf{t};g)&=\sum_{k=1}^{n-1}\tilde\varphi_k(\mathbf{t};g)\qquad\text{with}\quad
\tilde\varphi_k(\mathbf{t};g)=\sum_{I\in\mathcal{M}(Q_k)}\tilde\alpha_{I}^k(\mathbf{t};g)\,\mathsf{L}(I),\\ \nonumber
&\text{and} \qquad \tilde\alpha_I^k(\mathbf{t}; g)
=\frac{\left(\Delta_{I}(Q_ke^{\Theta_{\Lambda}(\mathbf{t})})\right)^2}
{\sum_{J\in\mathcal{M}(Q_k)}\left(\Delta_J(Q_k e^{\Theta_{\Lambda}(\mathbf{t})})\right)^2}.
\end{align*}
Note here that $\tilde\varphi(\mathbf{t};g)=\mu_k(Q_ke^{\Theta_{\Lambda}(\mathbf{t})})$.  Then  Lemma \ref{lemma:AQ} implies that 
$\tilde\alpha_I^k(\mathbf{t};g)=\alpha_I^k(\mathbf{t};g)$ for all $I\in\mathcal{M}(Q_k)=\mathcal{M}(A_k)$ and $k=1,\ldots,n-1$.  That is, the moment polytope 
for the full symmetric Toda hierarchy is exactly the same as that of the f-KT hierarchy.   The present theorem is then reduced to Theorem \ref{thm:polytope}.\end{proof}

\begin{remark}
It is natural to wonder what happens if instead of using solutions 
coming from points $gB^+$ of the tnn flag variety, we use solutions
coming from arbitrary points $gB^+$ of the real flag variety.
In this more general case, we no longer have the result
of Proposition \ref{prop:regular} which guarantees that the f-KT flows
are regular.  However, we may still 
have a unique QR-factorization of $g$ with positive diagonal entries in $r_0$ ($g=q_0r_0$),
and the flow of the full symmetric Toda lattice with the initial data
$\mathcal{L}^0=q_0^T\Lambda q_0$ \emph{is} regular.
This can be seen from the positivity of the $\tau$-functions defined by
\begin{align*}
\tau^{sym}_k(t)&=\left[\exp(2t\mathcal{L}^0)\right]_k=\left[ q_0^Te^{2t\Lambda}q_0\right]_k=\sum_{I\in\mathcal{M}(Q_k)}\Delta_I(Q_k)^2e^{2\theta_I(t)},
\end{align*}
where $Q_k$ is the submatrix consisting of the first $k$ columns of $q_0$, and
$\theta_I(t)=\sum_{j=1}^k\lambda_{i_j}t$ with $I=\{i_1,\ldots,i_k\}$. 
Since the f-KT flows are in general not regular but the full symmetric Toda flows
\emph{are} regular,
there is no canonical way to map flows of the f-KT lattice to the full symmetric Toda lattice.
\end{remark}

\appendix
\section{Bruhat interval polytopes}

In this section we study some basic combinatorial properties of the 
Bruhat interval polytopes that arose in Section \ref{sec:moment}.  A 
more extensive study will be taken up in a future work.
Recall 
that for $u \leq v \in \Sym_n$, the Bruhat interval polytope 
$\mathsf{P}_{u,v}$ is defined as
the convex hull of the permutation vectors $z$ such that 
$u \leq z \leq v$.
We show that each Bruhat interval polytope is a Minkowski sum
of some matroid polytopes, and is also a 
\emph{generalized permutohedron}, as defined by Postnikov
\cite{Postnikov2}.  Moreover, every edge of a Bruhat 
interval polytope corresponds to a cover relation in Bruhat order.

In order to make this section self-contained, we will 
recall the definitions of the various polytopes we need.  
To be more consistent with the combinatorial literature,
we will define our polytopes in $\R^n$ (rather than in the weight 
space); however, this does not affect any of our statements
about the polytopes.  We start by defining a \emph{matroid}
using the \emph{basis axioms}.

\begin{definition}
Let $\M$ be a nonempty
collection 
of $k$-element subsets of $[n]$ such that:
if $I$ and $J$ are distinct members of $\M$ and 
$i \in I \setminus J$, then there exists an element $j \in J \setminus I$
such that $I \setminus \{i\} \cup \{j\} \in \M$.
Then $\M$ is called the \emph{set of bases of a matroid} of \emph{rank $k$}
on the \emph{ground set} $[n]$; 
or simply a \emph{matroid}. 
\end{definition}

\begin{definition}
Given the set of bases $\M \subset {[n] \choose k}$ 
of a matroid, the \emph{matroid polytope} 
$\Gamma_{\M}$ of $\M$ is the convex hull of the indicator vectors of the bases of $\M$:
\[
\Gamma_{\M} := \convex\{e_I \mid I \in \M\} \subset \R^n,
\]
where $e_I := \sum_{i \in I} e_i$, and $\{e_1, \dotsc, e_n\}$ is the standard basis of $\R^n$.
\end{definition}

Note that ``a matroid polytope" refers to the polytope of a specific matroid in its specific position in $\R^n$.

Recall that any element $A \in Gr_{k,n}$ gives rise to a matroid
$\M(A)$ of rank $k$ on the ground set $[n]$: specifically, the bases 
of $\M(A)$ are precisely the $k$-element subsets $I$ such that 
$\Delta_I(A) \neq 0$.

\begin{definition}
The \emph{usual permutohedron} $\Perm_n$ in $\R^n$ is the convex
hull of the $n!$ points obtained by permuting the coordinates 
of the vector $(1,2,\dots,n)$.
\end{definition}

There is a beautiful generalization of permutohedra due to 
Postnikov \cite{Postnikov2}.  

\begin{definition} A 
\emph{generalized permutohedron} is a polytope which is 
obtained by moving the vertices
of the usual permutohedron in such a way that directions of 
edges are preserved, but some edges (and higher dimensional faces)
may degenerate.  
\end{definition}

The main topic of this section is Bruhat interval polytopes.

\begin{definition}
Let $v$ and $w \in \Sym_n$ such that $v \leq w$ in (strong) Bruhat 
order.  We identify each permutation $z\in \Sym_n$ with 
the corresponding vector $(z(1),\dots, z(n)) \in \R^n.$  
Then the \emph{Bruhat interval polytope}
$\mathsf{P}_{v,w}$ is defined as the convex hull of all 
vectors $(z(1),\dots,z(n))$ for $z$ such that $v \leq z \leq w$.
\end{definition}

See Figure \ref{fig:Mpoly} for some examples of Bruhat interval 
polytopes.

\begin{theorem}\label{prop:Bruhat-matroid}
Let $v, w \in \Sym_n$ such that $v \leq w$.  Then 
the Bruhat interval polytope $\mathsf{P}_{v,w}$ is the Minkowski sum of $n-1$ matroid
polytopes $\mathsf{P}_1,\dots,\mathsf{P}_{n-1}$.  
Specifically, if we choose any reduced expression
$\w$ for $w$, and the corresponding positive expression 
$\v_+$ for $v$ in $\w$, and choose any $g\in G_{\v_+,\w}^{>0}$
then $\mathsf{P}_k$ is the matroid polytope associated to the 
matroid  $\M(\pi_k(g))$.  See \eqref{eq:G} and \eqref{eq:pi} for the 
definitions of $G_{\v_+,\w}^{>0}$
and $\pi_k$.
\end{theorem}

\begin{proof}
This result is equivalent to Corollary \ref{cor:BIPM}.
\end{proof}

\begin{remark}
In fact $\mathsf{P}_{v,w}$ is the Minkowski sum of $n-1$ 
\emph{positroid polytopes}, which were recently studied  in \cite{ARW}.
\end{remark}

\begin{corollary}
Every Bruhat interval polytope is a generalized permutohedron.
\end{corollary}
\begin{proof}
This follows from the fact that matroid polytopes are generalized 
permutohedra (see \cite[Section 2]{ABD}), and the Minkowski sum of two 
generalized permutohedra is again a generalized permutohedron
(see \cite[Lemma 2.2]{ABD}).
\end{proof}

Our next result is about edges of Bruhat interval polytopes.
It generalizes the well-known fact that edges of $\Perm_n$
correspond precisely to the pairs of permutations $y$ and $z$ which are adjacent
in the \emph{weak Bruhat order}.  That is, we can write $z = y t$ where $t$
is an adjacent transposition $t = (i, i+1)$ for some $i$, and 
$\ell(z) = \ell(y)+1$.

Before stating our generalization of this result, we first recall the elegant characterization
of edges of matroid polytopes, due to Gelfand, Goresky, MacPherson,
and Serganova \cite{GGMS}.

\begin{theorem}[\cite{GGMS}]\label{r:GS} Let $\M$ be a collection of subsets of $[n]$ and let
$\Gamma_{\M}:=\convex\{e_I \mid I \in \M\} \subset \R^n$. Then $\M$ is the collection of bases of a matroid if and only if every edge of  $\Gamma_{\M}$ is a parallel translate of $e_i-e_j$ for some $i,j \in [n]$.
\end{theorem}

\begin{theorem}
Let $\mathsf{P}_{u,v}$ be a Bruhat interval polytope such that 
$u,v \in \Sym_n$.  Then 
every edge of $\mathsf{P}_{u,v}$ connects two vertices
$y,z \in \Sym_n$ such that $y \lessdot z$ is a cover relation 
in the (strong) Bruhat
order, that is, $y < z$ and there is no $w$ such that 
$y < w < z$.
\end{theorem}

\begin{proof}
From Theorem \ref{prop:Bruhat-matroid}, we have that every Bruhat interval polytope
$\mathsf{P}_{u,v}$ (for $u \leq v \in \Sym_n)$ is the Minkowski sum of $n-1$ matroid
polytopes $\mathsf{P}_1,\dots,\mathsf{ P}_{n-1}$.  It follows that each edge $E$ of
$\mathsf{P}_{u,v}$ is the Minkowski sum of an edge of one of those matroid
polytopes $\mathsf{P}_i$ and $n-2$ vertices (one from each of the other matroid polytopes).
Therefore by Theorem \ref{r:GS}, the vertices of $E$ must have the form
$y = (y(1), y(2), \dots, y(n))$ and $z = (z(1), z(2), \dots, z(n))$ where
$(z(1), z(2),\dots,z(n)) = (y(1), y(2),\dots,y(n)) + e_i - e_j$
for some $i,j$.
Therefore $z(i) = y(i)+1$ and $z(j) = y(j)-1$.
Since $(y(1), y(2), \dots, y(n))$ and
$(z(1), z(2), \dots, z(n))$ are both permutations, it follows that
$z = y t$ where $t$ is the transposition $(i,j)$, and
$\ell(z) = \ell(y)+1$, where
$\ell(z)$ is the number of inversions of $z$.
Therefore $y \lessdot z$ in Bruhat order.
\end{proof}



\raggedright

\bibliographystyle{amsalpha}
\bibliography{bibliography}

\def\cprime{$'$} \def\cprime{$'$}
\providecommand{\bysame}{\leavevmode\hbox to3em{\hrulefill}\thinspace}
\providecommand{\MR}{\relax\ifhmode\unskip\space\fi MR }
\providecommand{\MRhref}[2]{%
  \href{http://www.ams.org/mathscinet-getitem?mr=#1}{#2}
}
\providecommand{\href}[2]{#2}
\begin{thebibliography}{DLNT86b}

\bibitem[ABD10]{ABD}
F.~Ardila, C.~Benedetti, and J.~Doker, \emph{Matroid polytopes and their
  volumes}, Discrete Comput. Geom. \textbf{43} (2010), no.~4, 841--854.
  \MR{2610473 (2012b:52026)}

\bibitem[ARW13]{ARW}
F.~Ardila, F.~Rincon, and L.~Williams, \emph{Positroids and non-crossing
  partitions}, preprint, {\tt arXiv:1308.2698}, 2013.

\bibitem[Ati82]{A:82}
M.~F. Atiyah, \emph{Convexity and commuting {H}amiltonians}, Bull. London Math.
  Soc. \textbf{14} (1982), no.~1, 1--15. \MR{642416 (83e:53037)}

\bibitem[AvM99]{AvM}
M.~Adler and P.~van Moerbeke, \emph{Vertex operator solutions to the discrete
  {KP}-hierarchy}, Comm. Math. Phys. \textbf{203} (1999), no.~1, 185--210.
  \MR{1695172 (2000d:37097)}

\bibitem[BG98]{BG}
A.~M. Bloch and M.~I. Gekhtman, \emph{Hamiltonian and gradient structures in
  the {T}oda flows}, J. Geom. Phys. \textbf{27} (1998), no.~3-4, 230--248.
  \MR{1645028 (99i:58070)}

\bibitem[BK03]{BK}
G.~Biondini and Y.~Kodama, \emph{On a family of solutions of the
  {K}adomtsev-{P}etviashvili equation which also satisfy the {T}oda lattice
  hierarchy}, J. Phys. A \textbf{36} (2003), no.~42, 10519--10536. \MR{2024910
  (2005j:37128)}

\bibitem[CK02]{CK}
L.~Casian and Y.~Kodama, \emph{Blow-ups of the {T}oda lattices and their
  intersections with the {B}ruhat cells}, The legacy of the inverse scattering
  transform in applied mathematics ({S}outh {H}adley, {MA}, 2001), Contemp.
  Math., vol. 301, Amer. Math. Soc., Providence, RI, 2002, pp.~283--310.
  \MR{1947372 (2004j:37107)}

\bibitem[CSS12]{CSS}
Yu.~B. Chernyakov, G.~I. Sharygin, and A.~S. Sorin, \emph{Bruhat order in full
  symmetric {T}oda system}, preprint, {\tt arXiv:1212.4803}, 2012.

\bibitem[Deo85]{Deodhar}
V.~V. Deodhar, \emph{On some geometric aspects of {B}ruhat orderings. {I}. {A}
  finer decomposition of {B}ruhat cells}, Invent. Math. \textbf{79} (1985),
  no.~3, 499--511. \MR{782232 (86f:20045)}

\bibitem[DLNT86a]{DNT}
P.~Deift, L.~C. Li, T.~Nanda, and C.~Tomei, \emph{The {T}oda flow on a generic
  orbit is integrable}, Comm. Pure Appl. Math. \textbf{39} (1986), no.~2,
  183--232. \MR{820068 (87c:58045)}

\bibitem[DLNT86b]{DLNT}
\bysame, \emph{The {T}oda flow on a generic orbit is integrable}, Comm. Pure
  Appl. Math. \textbf{39} (1986), no.~2, 183--232. \MR{820068 (87c:58045)}

\bibitem[EFS93]{EFS}
N.~M. Ercolani, H.~Flaschka, and S.~Singer, \emph{The geometry of the full
  {K}ostant-{T}oda lattice}, Integrable systems ({L}uminy, 1991), Progr. Math.,
  vol. 115, Birkh\"auser Boston, Boston, MA, 1993, pp.~181--225. \MR{1279823
  (95g:58103)}

\bibitem[Ehr34]{Ehresmann}
C.~Ehresmann, \emph{Sur la topologie de certains espaces homog\`enes}, Ann. of
  Math. (2) \textbf{35} (1934), no.~2, 396--443. \MR{1503170}

\bibitem[FH91]{FH}
H.~Flaschka and L.~Haine, \emph{Vari\'et\'es de drapeaux et r\'eseaux de
  {T}oda}, Math. Z. \textbf{208} (1991), no.~4, 545--556. \MR{1136474
  (93e:58080)}

\bibitem[GGMS87]{GGMS}
I.~M. Gel{\cprime}fand, R.~M. Goresky, R.~D. MacPherson, and V.~V. Serganova,
  \emph{Combinatorial geometries, convex polyhedra, and {S}chubert cells}, Adv.
  in Math. \textbf{63} (1987), no.~3, 301--316. \MR{877789 (88f:14045)}

\bibitem[GS82]{GuS}
V.~Guillemin and S.~Sternberg, \emph{Convexity properties of the moment
  mapping}, Invent. Math. \textbf{67} (1982), no.~3, 491--513. \MR{664117
  (83m:58037)}

\bibitem[GS87]{GS}
I.~M. Gel{\cprime}fand and V.~V. Serganova, \emph{Combinatorial geometries and
  the strata of a torus on homogeneous compact manifolds}, Uspekhi Mat. Nauk
  \textbf{42} (1987), no.~2(254), 107--134, 287. \MR{898623 (89g:32049)}

\bibitem[KM96]{KM}
Y.~Kodama and K.~T-R McLaughlin, \emph{Explicit integration of the full
  symmetric {T}oda hierarchy and the sorting property}, Lett. Math. Phys.
  \textbf{37} (1996), no.~1, 37--47. \MR{1392145 (97d:58098)}

\bibitem[KS08]{KS08}
Y.~Kodama and B.~A. Shipman, \emph{The finite non-periodic toda lattice: a
  geometric and topological viewpoint}, preprint, {\tt arXiv:0805.1389}, 2008.

\bibitem[KW13]{KW3}
Y.~Kodama and L.~K. Williams, \emph{The {D}eodhar decomposition of the
  {G}rassmannian and the regularity of {KP} solitons}, Adv. Math. \textbf{244}
  (2013), 979--1032.

\bibitem[KY96]{KY}
Y.~Kodama and J.~Ye, \emph{Iso-spectral deformations of general matrix and
  their reductions on {L}ie algebras}, Comm. Math. Phys. \textbf{178} (1996),
  no.~3, 765--788. \MR{1395214 (97d:58097)}

\bibitem[Lus94]{Lusztig3}
G.~Lusztig, \emph{Total positivity in reductive groups}, Lie theory and
  geometry, Progr. Math., vol. 123, Birkh\"auser Boston, Boston, MA, 1994,
  pp.~531--568. \MR{1327548 (96m:20071)}

\bibitem[Lus98]{Lusztig2}
\bysame, \emph{Total positivity in partial flag manifolds}, Represent. Theory
  \textbf{2} (1998), 70--78. \MR{1606402 (2000b:20060)}

\bibitem[MR04]{MR}
R.~J. Marsh and K.~C. Rietsch, \emph{Parametrizations of flag varieties},
  Represent. Theory \textbf{8} (2004), 212--242 (electronic). \MR{2058727
  (2005c:14061)}

\bibitem[Pos]{Postnikov}
A.~Postnikov, \emph{Total positivity, {G}rassmannians, and networks}, Preprint.
  Available at \url{http://www-math.mit.edu/~apost/papers/tpgrass.pdf}.

\bibitem[Pos09]{Postnikov2}
A.~Postnikov, \emph{Permutohedra, associahedra, and beyond}, Int. Math. Res.
  Not. IMRN (2009), no.~6, 1026--1106. \MR{2487491 (2010g:05399)}

\bibitem[Rie98]{Rietsch}
K.~C. Rietsch, \emph{Total positivity and real flag varieties}, ProQuest LLC,
  Ann Arbor, MI, 1998, Thesis (Ph.D.)--Massachusetts Institute of Technology.
  \MR{2716793}

\bibitem[Rie06]{RietschClosure}
\bysame, \emph{Closure relations for totally nonnegative cells in {$G/P$}},
  Math. Res. Lett. \textbf{13} (2006), no.~5-6, 775--786. \MR{2280774
  (2007j:14073)}

\bibitem[Shi02]{Shipman}
B.~A. Shipman, \emph{Nongeneric flows in the full {K}ostant-{T}oda lattice},
  Integrable systems, topology, and physics ({T}okyo, 2000), Contemp. Math.,
  vol. 309, Amer. Math. Soc., Providence, RI, 2002, pp.~219--249. \MR{1953364
  (2005b:37127)}

\bibitem[Sym82]{Symes}
W.~W. Symes, \emph{The {$QR$} algorithm and scattering for the finite
  nonperiodic {T}oda lattice}, Phys. D \textbf{4} (1981/82), no.~2, 275--280.
  \MR{653781 (83h:58053)}

\bibitem[Tod89]{Toda}
M.~Toda, \emph{Theory of nonlinear lattices}, second ed., Springer Series in
  Solid-State Sciences, vol.~20, Springer-Verlag, Berlin, 1989. \MR{971987
  (89h:58082)}

\end{thebibliography}
\label{sec:biblio}

\end{document}